\documentclass[12pt,reqno]{amsart}
\usepackage{fullpage}
\usepackage{comment}
\usepackage{color}
\usepackage{enumerate}
\usepackage{ amssymb }
\usepackage{ amsthm}
\usepackage[normalem]{ulem}
\usepackage{thmtools}

\usepackage{amsfonts,txfonts,psfrag,color}
\usepackage[dvips]{graphicx}
\usepackage{epstopdf}
\usepackage{caption}
\usepackage[dvipsnames]{xcolor}
\usepackage[mathscr]{euscript}

\numberwithin{equation}{section}

\makeatletter
\renewcommand{\@secnumfont}{\bfseries}
\makeatother

\newtheorem{thm}{Theorem}[section]    
\newtheorem{lem}[thm]{Lemma}         
\newtheorem{prop}[thm]{Proposition}        
\newtheorem{coro}[thm]{Corollary}
\newtheorem{conj}[thm]{Conjecture}          
\theoremstyle{definition}
\newtheorem{defn}[thm]{Definition}  

\newtheorem{rmk}[thm]{Remark}

\newcommand{\N}{\mathbb{N}}
\newcommand{\Z}{\mathbb{Z}}

\newcommand{\R}{\mathbb{R}}

\newcommand{\prob}{\mathbb{P}}

\newcommand{\B}{\textbf}

\newcommand{\e}{\epsilon}
\newcommand{\al}{\alpha}
\newcommand{\vp}{\varphi}
\newcommand{\cal}{\mathcal}
\newcommand{\ra}{\rangle}
\newcommand{\la}{\langle}

\newcommand{\til}{\tilde}

\newcommand{\pa}{\partial}

\newcommand{\om}{\omega}

\newcommand{\var}{{\rm \mathbb{V}ar}}
\newcommand{\Th}{\text{Th\'{e}ret}}

\newcommand{\wh}{\widehat}
\newcommand{\Chee}{\Phi_{n}}

\newcommand{\Leb}{{\rm{Leb}}}

\newcommand{\giant}{\text{\rm\B{C}}_n}

\newcommand{\per}{\text{\rm per}}

\newcommand{\vol}{\text{\rm \small \textsf{vol}}}

\newcommand{\hull}{\text{\rm \small \textsf{hull}}}

\newcommand{\poly}{\text{\rm \small \textsf{poly}}}

\newcommand{\dper}{\text{\rm \small \textsf{d-per}}}

\newcommand{\sfS}{{\text{\rm\textsf{S}}}}

\newcommand{\wt}{\widetilde}

\newcommand{\image}{\text{\rm \small \textsf{image}}}

\newcommand{\len}{\text{\rm length}}
\newcommand{\Haus}{{\rm{d}}_H}
\newcommand{\tension}{\cal{I}_p}
\newcommand{\cluster}{\text{\rm \B{C}}_\infty}

\definecolor{azure(colorwheel)}{rgb}{0.0, 0.5, 1.0}
\definecolor{hanpurple}{rgb}{0.32, 0.09, 0.98}
\definecolor{iris}{rgb}{0.35, 0.31, 0.81}

\begin{document}
\title{Intrinsic isoperimetry of the giant component\\ of supercritical bond percolation in dimension two}
\author{Julian Gold}

\maketitle

\begin{abstract} We study the isoperimetric subgraphs of the giant component $\giant$ of supercritical bond percolation on the square lattice. These are subgraphs of $\giant$ having minimal edge boundary to volume ratio. In contrast to the work of \cite{BLPR}, the edge boundary is taken only within $\giant$ instead of the full infinite cluster. The isoperimetric subgraphs are shown to converge almost surely, after rescaling, to the collection of optimizers of a continuum isoperimetric problem emerging naturally from the model. We also show that the Cheeger constant of $\giant$ scales to a deterministic constant, which is itself an isoperimetric ratio, settling a conjecture of Benjamini in dimension two. 
\end{abstract}

\vspace{5mm}

\newpage

\tableofcontents 

\newpage

{\large\section{\B{Introduction and results}}\label{sec:introduction}}

Isoperimetric problems, while among the oldest in mathematics, are fundamental to modern probability and PDE theory. The goal of an isoperimetric problem is to characterize sets of minimal boundary measure subject to an upper bound on the volume measure of the set. The Cheeger constant, first introduced in Cheeger's thesis \cite{Cheeger} in the context of manifolds, is a way of encoding such problems. Alon \cite{Alon} later introduce the Cheeger constant for (finite) graphs $G$ as the following minimum over subgraphs of $G$:
\begin{align}
\Phi_G := \min \left\{ \frac{ | \pa H |}{ |H|} \::\: H \subset G,\, 0 < |H| \leq |G| /2 \right\} \,,
\label{eq:chee_original}
\end{align}
Here $\pa H$ is the edge boundary of $H$ in $G$ (the edges of $G$ having exactly one endpoint vertex in $H$), $|\pa H|$ denotes the cardinality of this set, and $|H|$ denotes the cardinality of the vertex set of $H$. The Cheeger constant of a graph $G$ measures the robustness of $G$; it provides information about the behavior of random walks on $G$ and is involved in a fundamental estimate in spectral graph theory (see Chapter 2 of \cite{FanChung}). This paper is concerned with the isoperimetric properties of random graphs arising from bond percolation in $\Z^2$.

Bond percolation is defined as follows: We view $\Z^2$ as a graph with standard nearest-neighbor graph structure and form the probability space $( \{0,1\}^{\rm{E}(\Z^2)}, \mathscr{F}, \prob_p)$ for the \emph{percolation parameter} $p \in [0,1]$. Here $\mathscr{F}$ denotes the product $\sigma$-algebra on $\{0,1\}^{\rm{E}(\Z^2)}$ and $\prob_p$ is the product Bernoulli measure associated to $p$. Elements of this probability space are written as $\om = (\om_e)_{e \in \rm{E}(\Z^2)}$ and are referred to as \emph{percolation configurations}. An edge $e$ is \emph{open} in the configuration $\om$ if $\om_e =1$, and is \emph{closed} otherwise. For each configuration $\om$, the collection of edges which are open in $\om$ determines a subgraph of $\Z^2$, written as $[\Z^2]^\om$. Under the probability measure $\prob_p$, $[\Z^2]^\om$ is then a random subgraph of $\Z^2$.

 Connected components of $[\Z^2]^\om$ are called \emph{open clusters}, or just \emph{clusters}. It is well known (Grimmett \cite{Grimmett} is a standard reference) that bond percolation on $\Z^2$ exhibits a phase transition: there is $p_c(2) \in (0,1)$ so that $p > p_c(2)$ implies there is a unique infinite open cluster $\prob_p$-almost surely, and such that $p < p_c(2)$ implies there is no infinite open cluster $\prob_p$-almost surely. Moreover, it is well known \cite{Kesten_Theorem} that $p_c(2) = 1/2$. We focus our attention on the supercritical ($p >p_c(2)$) regime, and let $\cluster = \cluster(\om)$ denote the unique infinite cluster which exists $\prob_p$-almost surely in this case. For $p > p_c(2)$, the quantity $\theta_p := \prob_p( 0 \in \cluster)$ is positive, and is referred to as the \emph{density} of $\cluster$ within~$\Z^2$.\\

\subsection{\small\textsf{A conjecture}}  It is possible to study the geometry of $\cluster$ using the Cheeger constant: define $\wt{\B{C}}_n := \cluster \cap [-n,n]^2$, and define the \emph{giant component} $\giant$ to be the largest connected component of $\wt{\B{C}}_n$. The random variable $\Chee := \Phi_{\giant}$ is central to this paper. It is known (Benjamini and Mossel \cite{BenjMo}, Mathieu and Remy \cite{Mathieu_Remy}, Rau \cite{Rau}, Berger, Biskup, Hoffman and Kozma \cite{BBHK} and Pete \cite{Pete}) that $\Chee \asymp n^{-1}$ as $n \to \infty$, prompting the following conjecture of Benjamini, which we state in all dimensions $d\geq2$. 

\begin{conj} (Benjamini) Let $d \geq 2$ and $p > p_c(d)$. The limit 
\begin{align}
\lim_{n \to \infty} n \Phi_{\giant}
\end{align}
exists $\prob_p$-almost surely as a deterministic constant in $(0,\infty)$.
\label{conj:benjamini}
\end{conj}

Procaccia and Rosenthal \cite{Procaccia_Rosenthal} showed for $d\geq 2$ that $\var(n \Chee ) \leq cn^{2-d}$ for some positive $c(p,d)$. Biskup, Louidor, Procaccia and Rosenthal \cite{BLPR} settled Conjecture \ref{conj:benjamini} in $d=2$ for a natural modification $\wt{\Phi}_n$ of the Cheeger constant. The results of \cite{BLPR} go beyond resolving Conjecture \ref{conj:benjamini} for $\wt{\Phi}_n$: the random variables $\wt{\Phi}_n$ encode a sequence of discrete, random isoperimetric problems, whose set of optimizers are the subgraphs of $\wt{\B{C}}_n$ realizing the minimum defining $\wt{\Phi}_n$. The main result of \cite{BLPR} is that these optimizers, upon rescaling, almost surely tend (with respect to Hausdorff distance) to a translate of a deterministic shape, a convex subset of $[-1,1]^2$ whose two-dimensional Lebesgue measure is half that of $[-1,1]^2$. This limit shape, known as the \emph{Wulff shape} and denoted $W_p$, is the solution to a deterministic isoperimetric problem, defined in the continuum for rectifiable subsets of $[-1,1]^2$. 

We settle Conjecture \ref{conj:benjamini} for the original Cheeger constant $\Chee$ by employing the same overall strategy of \cite{BLPR}. The distinction between $\Chee$ and the modified Cheeger constant $\wt{\Phi}_n$ is that, in the latter object, the edge boundary of a subgraph $H \subset \giant$ is taken in the full infinite cluster $\cluster$ instead of just $\giant$. This modification simplifies the nature of the limiting isoperimetric problem, which is the analogue of the standard Euclidean isoperimetric problem for an anisotropic perimeter functional. In our case, a \emph{restricted perimeter} functional replaces the perimeter functional, reflecting the fact that $\Chee$ does not ``see" edges outside the box $[-n,n]^2$. 

\subsection{\small\textsf{The general form of the limiting variational problem}} A \emph{curve} $\lambda$ in the unit square $[-1,1]^2$ is the image of a continuous function $\lambda : [0,1] \to [-1,1]^2$. A curve $\lambda$ is \emph{closed} if $\lambda(0) = \lambda(1)$ in any parametrization, \emph{Jordan} if it is closed and one-to-one on $[0,1)$ and \emph{rectifiable} if there is a parametrization of $\lambda$ such that 
\begin{align}
\len(\lambda) := \sup_{n \in \N} \sup_{t_1 < \dots < t_n \in [0,1]} \sum_{j=1}^{n} | \lambda(t_j) - \lambda(t_{j-1}) |_2 < \infty \,.
\end{align}
Many of the curves considered in this paper will be Jordan, and we will thus often conflate a curve $\lambda$ with its image, denoted $\image(\lambda)$. In Section \ref{sec:variational}, we will study the variational problem \eqref{eq:2_variational_problem} defined below in greater detail, and there we will be more careful. The class $\cal{R}$ of sets we work with is defined as follows
\begin{align}
\cal{R} := \left\{ R \subset [-1,1]^2 \:: \begin{matrix} R \text{ is compact, } R^\circ \neq \emptyset,\, \pa R \text{ is a finite union of rectifiable Jordan} \\ \text{curves, and the intersection of any two such curves is $\cal{H}^1$-null} \end{matrix}\right\} \,,
\end{align}
where $\cal{H}^1$ denotes the one-dimensional Hausdorff measure, and where $R^\circ$ denotes the interior of $R$. Given a norm $\tau$ on $\R^2$, define the restricted perimeter functional $\cal{I}_\tau$ on elements of $\cal{R}$ via
\begin{align}
\cal{I}_\tau(\pa R) := \int_{\pa R \cap (-1,1)^2 } \tau(n_x) \cal{H}^1(dx) \,,
\label{eq:1_iso_problem_boundary}
\end{align}
where $n_x$ is the normal vector to $\pa R \cap (-1,1)^2$ which exists at $\cal{H}^1$-almost every point on the curves $\pa R \cap (-1,1)^2$. Given the functional $\cal{I}_\tau$, form the following variational problem, of central interest in this paper
\begin{align}
\text{minimize: }\, \frac{\cal{I}_\tau(\pa R)}{ \Leb(R)}\,, \hspace{5mm} \text{subject to: }\, \Leb(R) \leq 2 \,,
\label{eq:2_variational_problem}
\end{align}
where $R \in \cal{R}$, and where $\Leb$ is the two-dimensional Lebesgue measure.

\subsection{\small\textsf{Results}} Let $\cal{G}_n$ be the set of \emph{Cheeger optimizers}, the subgraphs of $\giant$ realizing the minimum defining $\Chee$. Recall that the Hausdorff metric on (non-empty) compact subsets of $[-1,1]^2$ is defined as follows: given $A, B \subset [-1,1]^2$ compact,
\begin{align}
{\rm d}_H(A,B) := \max \left( \sup_{x \in A} \inf_{y \in B} \big| x- y\big|_\infty,\, \sup_{y \in B} \inf_{x \in A} \big| x-y \big|_\infty \right) \,,
\label{eq:1_haus_metric}
\end{align}
where for $x, y \in \R^2$ and $p \in [1,\infty]$, $|x-y|_p$ denotes the $\ell^p$-distance between $x$ and $y$. The following shape theorem is the first of our main results. 

\begin{thm} Let $d=2$ and let $p >p_c(2)$. There is a norm $\beta_p$ on $\R^2$ with non-empty collection of optimizers $\cal{R}_p$ to the associated variational problem \eqref{eq:2_variational_problem} so that 
\begin{align}
\max_{G_n \in \cal{G}_n} \inf_{E \in \cal{R}_p} {\rm d}_H\Big( n^{-1} G_n\,,\, E  \Big) \xrightarrow[n \to \infty]{}  0
\end{align}
holds $\prob_p$-almost surely.
\label{thm:main_shape}
\end{thm}

The following definitions link Theorem \ref{thm:main_shape} with the limit in Conjecture \ref{conj:benjamini}.
\begin{defn} Let $\beta_p$ be the norm in Theorem \ref{thm:main_shape}, which is the norm defined in \cite{BLPR}. Given $R \in \cal{R}$, define the ratio
\begin{align}
\frac{ \cal{I}_{\beta_p}(\pa R) }{\Leb(R)}
\end{align}
to be the \emph{conductance} of $R$. Define the constant $\vp_p$ as 
\begin{align}
\vp_p := \inf \left\{ \frac{ \cal{I}_{\beta_p}(\pa R) }{\Leb(R)} \:: R \in \cal{R},\, \Leb(R) \leq 2 \right\}
\label{eq:1_optimal_conductance} \,.
\end{align}
\end{defn}

We remark that the two appearing in \eqref{eq:1_optimal_conductance} and \eqref{eq:2_variational_problem} is half the area of $[-1,1]^2$, and is an artifact of the 2 in the denominator of  \eqref{eq:chee_original}. Theorem \ref{thm:main_asym} below is the second of our main results and settles Conjecture \ref{conj:benjamini} in dimension two. 

\begin{thm} Let $d =2$ and let $p > p_c(2)$. Then $\prob_p$-almost surely,
\begin{align}
\lim_{n\to \infty} n \Chee = \frac{\vp_p}{\theta_p} \,,
\end{align}
where $\theta_p = \prob_p( 0 \in \cluster)$, and where $\vp_p \in (0,\infty)$ is defined in \eqref{eq:1_optimal_conductance}. 
\label{thm:main_asym}
\end{thm}

\begin{defn} For $U$ a subgraph of  $\giant$, let $\pa^nU$ denote the edge boundary of $U$ in $\giant$. We refer to this set as the \emph{open edge boundary of $U$ in $\giant$}. Let $\pa^\infty U$ denote the edge boundary of $U$ in all of $\cluster$, which we refer to as the \emph{open edge boundary of $U$}. Define the \emph{$n$-conductance} of $U$ to be the ratio
$| \pa^n U| / |U|$ and define the \emph{conductance} of $U$ to be the ratio $|\pa^\infty U| / |U|$. 
\end{defn}

\begin{rmk} Theorem \ref{thm:main_shape} says that the optimizers to the variational problems encoded by the $\Chee$ scale to the optimizers of \eqref{eq:2_variational_problem} for $\tau = \beta_p$. The random variable $\Chee$ is the $n$-conductance of any $G_n \in \cal{G}_n$. Theorem \ref{thm:main_asym} says that these $n$-conductances scale to the optimal conductance \eqref{eq:1_optimal_conductance} of the continuum problem \eqref{eq:2_variational_problem}  for the norm $\beta_p$. 
\label{rmk:1_thm_link}
\end{rmk}

\subsection{\small\textsf{Outline}} In Section \ref{sec:norm}, we recall the definition of $\beta_p$ from \cite{BLPR}, and we reintroduce the notion of right-most paths, which are used to define $\beta_p$. We collect the useful properties of both the norm and right-most paths. 

In Section \ref{sec:variational}, we study the variational problem \eqref{eq:2_variational_problem} for $\tau = \beta_p$. The main results here are existence and stability results: we first show the set $\cal{R}_p$ of optimizers of this problem is non-empty. We then show that if a connected set $R \in \cal{R}$ is $\Haus$-far from $\cal{R}_p$, the conductance of $R$ is at least $\vp_p$ plus a positive constant depending on the distance of $R$ to $\cal{R}_p$. 

In Section \ref{sec:upper}, we show the conductance of $R \in \cal{R}$ with $\Leb(R) \leq 2$ gives rise to upper bounds on $\Chee$ with high probability. Specifically, we pass from a nice object in the continuum to a subgraph of $\giant$, and we relate the conductances of these two objects. We do this first for polygons, and then for more general sets, making use of the tools collected in Section \ref{sec:norm}.  Ultimately, we show that for any $\e >0$, we have $n\Chee \leq (1+\e)\vp_p$ with high probability. 

In Section \ref{sec:disc_to_cts}, we move in the other direction, extracting from each Cheeger optimizer $G_n \in \cal{G}_n$ $R \in \cal{R}$ with $\Haus(G_n, nR)$ small, and relating the conductances of these objects. By controlling $\Leb(R)$ from above, we see that the conductance of $R$ is at least $(1-\e) \vp_p$, which translates to a high probability lower bound on $\Chee$ of this form. This settles Theorem \ref{thm:main_asym}. We then use the stability result of Section \ref{sec:variational} with the main result of Section \ref{sec:upper} to see that it is rare for $G_n$ to be far from $\cal{R}_p$, settling Theorem \ref{thm:main_shape}.

\subsection{\small\textsf{Discussion and context}} We use many of the tools developed in \cite{BLPR}, and as such, our work can be seen as falling under the umbrella of the Wulff construction program. This was initiated in the early 1990s independently by Dobrushin, Koteck\'{y} and Shlosman \cite{DKS} in for the Ising model and by Alexander, Chayes and Chayes \cite{Alexander_Chayes_Chayes} in percolation, both on the square lattice. 

These works characterized the asymptotic shape of a large droplet of one phase of the model (for instance, a large finite open cluster in supercritical bond percolation). The probability of such an event decays rapidly in the size of the droplet, thus the theory of large deviations plays a crucial role in the analysis and is key to defining a model-dependent norm $\tau$. Though the large droplets are not the minimizers of any isoperimetric problem, their limit shape is the minimizer of
\begin{align}
\text{minimize: }\, \frac{\len_{\tau}(\pa R)}{ \Leb(R)}\,, \hspace{5mm} \text{subject to: }\, \Leb(R) \leq c \,
\label{eq:2_unrestricted_problem}
\end{align}
for some constant $c>0$. The solution to \eqref{eq:2_unrestricted_problem}, called the Wulff shape, is easily defined and was postulated by Wulff \cite{Wulff} in 1901; it is a convex subset of $\R^2$ depending on $\tau$. This solution is known to be unique up to translations and modifications on a null set thanks to the substantial work of Taylor \cite{T2,T3,T1}, whose results hold in all dimensions at least two. The Wulff construction has been successfully employed in dimensions strictly larger than two \cite{Cerf_3D, Bodineau1, Bodineau2, Cerf_Pisztora_2, Cerf_Pisztora}, though with significant technical overhead due to geometric complications arising in higher dimensions. More details can be found in Section 5.5 of \cite{stflour} and in \cite{BIV}. 

The present work, as well as that of \cite{BLPR}, differs from the above in that we work exclusively in an event of full probability, and that we are faced with a collection of isoperimetric problems even at the discrete level. In our case, the variational problem in the continuum is a limit of these discrete problems. Because we study the unmodified Cheeger constant, our limiting variational problem \eqref{eq:2_variational_problem} is more complicated than the variational problem given by \eqref{eq:2_unrestricted_problem}. The shapes of droplets in the presence of a boundary, a single infinite wall, have been studied in the context of the Ising model \cite{Pfister_Velenik_1, BIV_winter} using the analogue of the Wulff construction known as the Winterbottom construction \cite{Winterbottom}. This construction has been generalized further in a paper of Koteck\'{y} and Pfister \cite{Kotecky_Pfister}, and related problems have been studied by Schlosman \cite{Schlosman}.

\subsection{\small\textsf{Open problems}} We remark on several future directions:\\

\emph{\B{(1)}} We find it desirable to classify elements of $\cal{R}_p$ in terms of the Wulff shape $W_p$, the limit shape obtained in \cite{BLPR} and the solution to the unrestricted isoperimetric problem \eqref{eq:2_unrestricted_problem} for the norm $\beta_p$. Based on work of Koteck\'{y} and Pfister \cite{Kotecky_Pfister} and Schlosman \cite{Schlosman}, we conjecture that the collection $\cal{R}_p$ consists of quarter-Wulff shapes or their complements in the square. Answering such questions may require a better understanding of the regularity of the norm. Questions regarding the regularity and strict convexity of $\beta_p$ are interesting in their own right and are related to open problems in first-passage percolation (see for instance Chapter 2 of \cite{ADH}). \\

\emph{\B{(2)}} Instead of studying the largest connected component of $\cluster \cap [-n,n]^2$, we can fix a Jordan domain $\Omega \subset \R^2$ and consider the Cheeger constant of the largest connected component of $\cluster \cap n \Omega$. The argument in this paper is likely robust enough that both Cheeger asymptotics and a shape theorem can be deduced in this case (perhaps depending on the convexity of $\Omega$). This problem is similar in flavor to work of Cerf and $\Th$ \cite{CeTh}, in which the shapes of minimal cutsets in first passage percolation are studied for more general domains. \\

\emph{\B{(3)}} A sharp limit and related shape theorem were recently obtained \cite{G1} for the modified Cheeger constant in dimensions three and higher. It is likely that by combining the techniques of \cite{G1} and the present paper, one can prove analogues of Theorem \ref{thm:main_shape} and Theorem \ref{thm:main_asym} for the giant component in dimensions larger than two.

\subsection{\small\textsf{Acknowledgements}} I thank my advisor Marek Biskup for suggesting this problem, and for his guidance. I thank David Clyde, John Garnett, Stephen Ge, Nestor Guillen, Inwon Kim and Peter Petersen for useful conversations. This research has been partially supported by the NSF grant DMS-1407558.

{\large\section{\B{The boundary norm}}\label{sec:norm}}

The motivation for the construction of $\beta_p$ goes back to the postulate of Gibbs \cite{Gibbs} that one phase of matter immersed in another will arrange itself so that the surface energy between the two phases is minimized. By regarding each $G_n \in \cal{G}_n$ as a droplet immersed in $\giant \setminus G_n$, we can study the interface between these two ``phases" and attempt to extract a surface energy. 

Our tool for studying these interfaces are right-most paths, introduced in \cite{BLPR}. Each Cheeger optimizer $G_n$ may be expressed using finitely many right-most circuits, which together represent the boundary of $G_n$  and hence the total interface between $G_n$ and $\giant \setminus G_n$. We assign a weight to each right-most path which depends on the percolation configuration, so that the combined weight of all right-most paths making up the boundary of $G_n$ is exactly $|\pa^\infty G_n|$. 

Given $v \in \mathbb{S}^1$, the value $\beta_p(v)$ encodes the asymptotic minimal weight of a right-most path joining two vertices $x,y \in \Z^d$ with $y-x$ a large multiple of $v$. Thus, the norm $\beta_p$ encodes the surface energy minimization taking place locally at the boundary of each $G_n$.\\

\subsection{\small\textsf{Right-most paths}} Consider the graph $\Z^2 = (\rm{V}(\Z^2), \rm{E}(\Z^2))$. Given $x,y \in \rm{V}(\Z^2)$, a \emph{path from $x$ to $y$} is an alternating sequence of vertices and edges $\gamma = (x_0, e_1, x_1, \dots, e_m, x_m)$ such that $e_i$ joins $x_{i-1}$ with $x_i$ for $i \in \{1, \dots, m\}$, and such that $x_0 = x$ and $x_m = y$. The \emph{length} of $\gamma$, denoted $|\gamma|$, is $m$. If $x_0 = x_m$, the path is said to be a \emph{circuit}.

It is useful to regard edges in a given path $\gamma$ as oriented, so that the edge $e_i$ starting at $x_{i-1}$ and ending at $x_i$, denoted $\la x_{i-1}, x_i\ra$, is considered distinct from the edge starting at $x_i$ and ending at $x_{i-1}$, denoted $\la x_i, x_{i-1} \ra$. A path $\gamma$ in $\Z^2$ is \emph{simple} if no oriented edge is used twice. Given paths $\gamma_1 = (x_0, e_1, \dots, e_m, x_m)$ and $\gamma_2 = (y_0, f_1, \dots, f_k, y_k)$ with $x_m = y_0$, define the \emph{concatenation} of $\gamma_1$ and $\gamma_2$, denoted $\gamma_1 * \gamma_2$ to be the path $(x_0, e_1, \dots, e_m, x_m, f_1 \dots, f_k ,y_k)$. 

\begin{defn} Let $\gamma$ be a path in $\Z^d$ and let $x_i$ be a vertex in $\gamma$ with $x_{i-1}$ and $x_{i+1}$ well-defined. The \emph{right-boundary edges at $x_i$} are obtained by enumerating all oriented edges whcih start at $x_i$, beginning with but not including $\la x_i, x_{i-1} \ra$, proceeding in a counter-clockwise manner and ending with but not including $\la x_i, x_{i+1} \ra$. If either $x_{i-1}$ or $x_{i+1}$ is not well-defined, the right-most boundary edges at $x_i$ are defined to be the empty set. The \emph{right-boundary of $\gamma$}, denoted $\pa^+ \gamma$, is the union of all right-boundary edges at each vertex of $\gamma$. 
\end{defn}

\begin{defn} A path $\gamma = (x_0, e_1, x_1, \dots, e_m, x_m)$ is said to be \emph{right-most} if it is simple, and if no $e_i$ is an element of $\pa^+\gamma$. 
\end{defn}

\begin{figure}[h]
\centering
\includegraphics[scale=1.5]{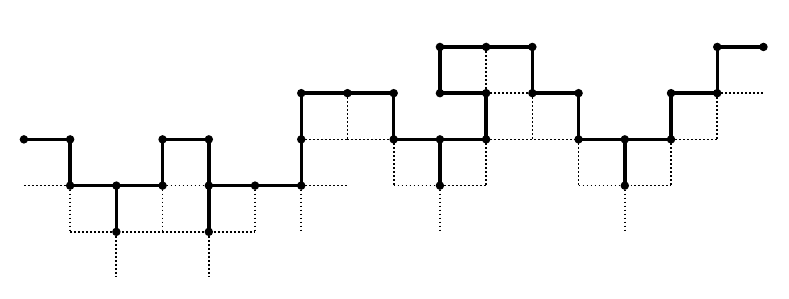}
\caption{In black, a right-most path which begins on the left and ends on the right. The dotted edges are the right-most boundary of this path.} 
\label{fig:right-most}
\end{figure}

\begin{defn}We assign configuration-dependent weights to right-most paths. Define the edge-sets
\begin{align}
\frak{b}(\gamma) &:= \big\{ e \in \pa^+\gamma \:: \om(e) \text{ is open} \big\}\,,\\
\frak{b}^n(\gamma) &:= \big\{ e \in \frak{b}(\gamma) \:: e \subset [-n,n]^2 \big\} \,,
\end{align}
and refer to $| \frak{b}(\gamma)|$ and $| \frak{b}^n(\gamma)|$ respectively as the \emph{$\cluster$-length} of $\gamma$ and the \emph{$\giant$-length} of $\gamma$.
\end{defn}
 
 \begin{rmk}
 As we will see in Lemma \ref{lem:5_circuit_decomp}, the boundary of a subgraph $U$ of $\giant$ may be expressed as a collection of right-most circuits. The total $\cluster$-length of these circuits will correspond to the size of $\pa^\infty U$, and the total $\giant$-length of these circuits will correspond to the size of $\pa^n U$. 
\label{rmk:2_circuit_talk}
\end{rmk}

Following \cite{BLPR}, we let $\cal{R}(x,y)$ denote the collection of all right-most paths joining $x$ to $y$. If vertices $x$ and $y$ are joined by an open path (and hence joined by an open right-most path) in the configuration $\om$, define the \emph{right-boundary distance} from $x$ to $y$ as
\begin{align}
b(x,y) := \inf \big\{ \frak{b}(\gamma) \:: \gamma \in \cal{R}(x,y),\, \gamma \text{ uses only open edges} \big\}\,.
\end{align}

\begin{rmk}It is convenient to allow $b$ to act on points in $\R^2$ by assigning to each $x \in \R^2$ a ``nearest" point $[x]$ in $\cluster$. To do this, we augment our probability space to support a collection $\{ \eta_x : x \in \Z^2 \}$ of iid random variables uniform on $[0,1]$ and independent of the Bernoulli random variables used to define the bond percolation. Given $x \in \R^2$, we let $[x]$ be the nearest (in $\ell^\infty$-sense) vertex in $\cluster$ to $x$, breaking ties using the $\eta_x$ if necessary. 
\end{rmk}

One can establish high-probability closeness of any $x \in \R^2$ with $[x]$ using a duality argument; the following is Lemma 2.7 of \cite{BLPR}.

\begin{lem} Suppose $p > p_c(2)$. There are positive constants $c_1(p), c_2(p)$ so that for all $x \in \Z^2$ and all $r >0$, 
\begin{align}
\prob_p \Big( | [x] - x |_2 > r \Big) \leq c_1 \exp \Big( -c_2 r \Big) \,.
\end{align}
\label{lem:2_close_cluster}
\end{lem}

\subsection{\small\textsf{Properties of right-most paths}} 

Before defining $\beta_p$, we mention some useful properties of right-most paths. In particular, we recall the correspondence between right-most paths and simple paths in the medial graph of $\Z^2$. Given a planar graph $G = (\rm{V}, \rm{E})$, the \emph{medial graph} $G_\sharp = (\rm{V}_\sharp, \rm{E}_\sharp)$ is the graph with vertices $\rm{V}_\sharp = \rm{E}$, and with any two vertices in $\rm{V}_\sharp$ adjacent in $G_\sharp$ if the corresponding edges of $G$ are adjacent in a face of $G$. 

An \emph{interface} is an edge self-avoiding oriented path in $\Z_\sharp^2$, which does not use its initial or terminal vertex more than once, except to close a circuit. There is a correspondence between interfaces and right-most paths: an interface $\pa = (e_1, \dots , e_m)$, written as a sequence of vertices in $\Z_\sharp^2$, either reflects on a given edge $e_i$ or cuts through a given edge.

\begin{figure}[h]
\centering
\includegraphics[scale=1.5]{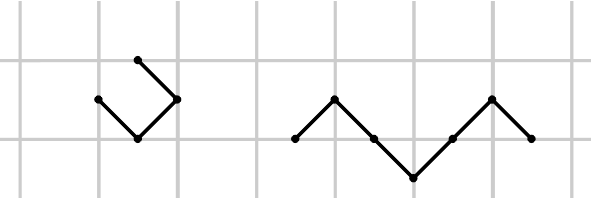}
\caption{The medial path of length three on the left reflects on each edge. On the right, the medial path of length six cuts through each edge.} 
\label{fig:reflect_cut_through}
\end{figure}

More rigorously, an interface $\pa = (e_1, \dots, e_m)$ is said to \emph{reflect} on $e_i$ (for $i \in \{2, \dots, m-1\}$) if $e_{i-1}$ and $e_{i+1}$ are on the boundary of the same face of $\Z^2$, and $\pa$ is said to \emph{cut through} $e_i$ otherwise. The following proposition (Proposition 2.3 of \cite{BLPR}) provides a fundamental correspondence between interfaces and right-most paths.

\begin{prop} For each interface $\pa = (e_1, \dots, e_m)$, the subsequence $(e_{k_1}, \dots, e_{k_n})$ of edges which are not cut through by $\pa$ forms a right-most path $\gamma$. This mapping is one-to-one and onto the set of all right-most paths. In particular, $\gamma$ is a right-most circuit if and only if $\pa$ is a circuit in the medial graph. Finally, the edges of $\pa \setminus ( e_{k_1}, \dots, e_{k_n})$ (oriented properly) form $\pa^+ \gamma$. 
\label{prop:2_interface_path}
\end{prop}

\begin{figure}[h]
\centering
\includegraphics[scale=1.5]{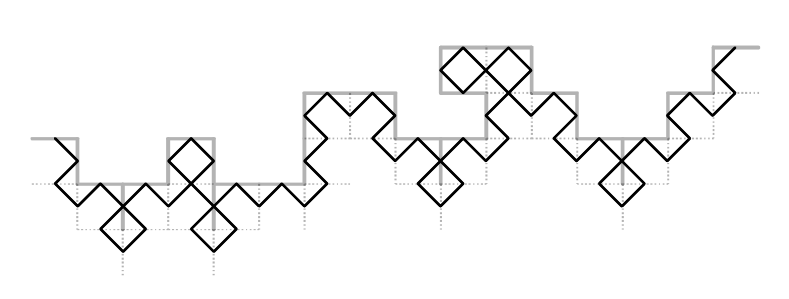}
\includegraphics[scale=1.5]{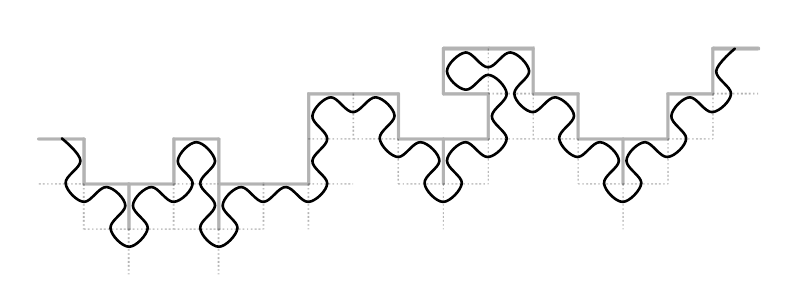}
\caption{Above: the correspondence of Proposition \ref{prop:2_interface_path}, built from the right-most path in Figure \ref{fig:right-most}. Below: the perturbed interface is a simple curve.} 
\label{fig:right-most_2}
\end{figure}

\begin{rmk} Interfaces may be perturbed via ``corner-rounding" to simple curves in $\R^2$, as illustrated at the bottom of Figure \ref{fig:right-most_2}. In particular, if $\gamma$ is a right-most circuit, it may be identified with a rectifiable Jordan curve $\lambda_\pa$ built from the interface $\pa$ corresponding to $\gamma$ via Proposition \ref{prop:2_interface_path}. 
\label{rmk:2_corner_round}
\end{rmk}

\begin{defn} Let $\lambda$ be a rectifiable curve and for $x \notin \lambda$, let $w_\lambda(x)$ denote the winding number of $\lambda$ around $x$. Define
\begin{align}
\hull(\lambda) := \lambda \cup \big\{ x \notin \lambda \:: w_\lambda(x) \text{ is odd} \big\}\,,
\label{eq:2_hull_def}
\end{align}
\end{defn}

A fundamental property of right-most circuits is that they may be used to ``carve out" subgraphs of $\giant$. This is done in a way which conveniently links the total length of the circuits with the edge boundary of the subgraph, see Remark \ref{rmk:2_circuit_talk}. Let $\cal{U}_n$ denote the collection of \emph{connected} subgraphs of $\cluster \cap [-n,n]^2$ determined by their vertex set. Given an interface $\pa$ corresponding to a right-most circuit, let $\lambda_\pa$ be the Jordan curve obtained from $\pa$ by rounding the corners, and write $\hull(\pa)$ for $\hull(\lambda_\pa)$. The following is proved by inducting on the size of the vertex set of $U$.

\begin{lem} Let $U \in \cal{U}_n$. The graph $\cluster \setminus U$ consists of a unique infinite connected component and finitely many finite connected components $\Lambda_1, \dots, \Lambda_m$. There are open right-most circuits $\gamma, \gamma_1, \dots, \gamma_m$ contained in $U$, where $\gamma$ is oriented counter-clockwise and each $\gamma_j$ is oriented clockwise so that 
\begin{enumerate}
\item $\pa, \pa_1, \dots, \pa_m$ are disjoint\,,
\item $\frak{b}(\gamma) \cup \left( \bigsqcup_{j=1}^m \frak{b}(\gamma_j) \right) = \pa^\infty U$\,,
\item $U = \left[ \hull(\pa) \setminus \left( \bigsqcup_{j=1}^m \hull(\pa_j) \right) \right] \cap \cluster$\,,
\item For each $j \in \{1, \dots, m\}$, we have $\Lambda_j = \hull(\pa_j) \cap \cluster$\,,
\end{enumerate}
where $\pa$ is the counter-clockwise interface corresponding to $\gamma$, and where each $\pa_j$ is the clockwise interface corresponding to $\gamma_j$. 
\label{lem:5_circuit_decomp}
\end{lem}

The final input on right-most paths we include is Proposition 2.9 of \cite{BLPR}, which tells us $|\gamma|$ and $|\frak{b}(\gamma)|$ are comparable when $|\gamma|$ is sufficiently large. This enables us to pass from discrete sets with reasonably sized open edge boundaries to rectifiable sets in the continuum.

\begin{prop} Let $p > p_c(2)$. There are positive constants $\al, c_1, c_2$ depending only on $p$ such that for all $n \geq 0$, we have
\begin{align}
\prob_p\Big( \exists \gamma \in \bigcup_{x \in \Z^2} \cal{R}(0,x) \:: |\gamma|\geq n , | \frak{b}(\gamma)| \leq \al n \Big) \leq c_1\exp(-c_2 n) \,.
\end{align}
\label{prop:2_length_comparison}
\end{prop}

\subsection{\small\textsf{The norm}}

We now use right-most paths to define the norm $\beta_p$ on $\R^2$, and we aggregate several useful results from \cite{BLPR}. The following is the main result (Theorem 2.1 and Proposition 2.2) of Section 2 in \cite{BLPR}, which we state verbatim.

\begin{thm} Let $p > p_c(2)$, and let $x \in \R^2$. The limit
\begin{align}
\beta_p(x) := \lim_{n \to \infty} \frac{ b( [0], [nx] )}{n} 
\end{align}
exists $\prob_p$-almost surely and is non-random, non-zero (when $x \neq 0$) and finite. The limit also exists in $L^1$ and the convergence is uniform on $\{ x \in \R^2 \:: |x|_2 = 1\}$. Moreover,

\begin{enumerate}
\item $\beta_p$ is homogeneous, i.e. $\beta_p(c x) = | c | \beta_p(x)$ for all $x \in \R^2$ and all $c \in \R$,
\item $\beta_p$ obeys the triangle inequality
\begin{align}
\beta_p(x + y) \leq \beta_p(x) + \beta_p(y) \,,
\end{align}
\item $\beta_p$ inherits the symmetries of $\Z^2$; for all $(x_1, x_2) \in \R^2$, we have
\begin{align}
\beta_p\big( (x_1, x_2) \big) = \beta_p \big( (x_2, x_1) \big) = \beta_p \big( ( \pm x_1, \pm x_2 ) \big) \,
\end{align}
for any choice of the signs $\pm$. 
\end{enumerate}
\label{thm:2_norm}
\end{thm}

\begin{rmk} Theorem \ref{thm:2_norm} tells us $\beta_p$ defines a norm on $\R^2$, and that this norm inherits the symmetries of $\Z^2$. 
It is first proved by appealing to the subadditive ergodic theorem, but can also be deduced from concentration estimates developed in Section 3 of \cite{BLPR}, which we state below. 
\end{rmk}

The first concentration estimate we record is measure theoretic, it is Theorem 3.1 of \cite{BLPR}. 

\begin{thm} Let $p > p_c(2)$. For each $\e >0$, there are positive constants $c_1(p,\e), c_2(p,\e)$ so that for all $x, y \in \Z^2$, 
\begin{align}
\prob_p\left( \left| \frac{ b( [x], [y]) }{\beta_p(y-x) } -1 \right| > \e \right) \leq c_1 \exp \Big( -c_2 \log^2 | y-x |_2 \Big) \,.
\end{align}
\label{thm:2_meas_con}
\end{thm}

We also require a result on the geometric concentration of right-most paths; namely that right-most paths which are almost optimal are geometrically close to the straight line joining their endpoints. Given $x, y \in \cluster$, say that $\gamma \in \cal{R}(x,y)$ is \emph{$\e$-optimal} if 
\begin{align}
\frak{b}(\gamma) - b(x,y) \leq \e |y -x|_2 \,,
\label{eq:e_optimal}
\end{align}
and write $\Gamma_\e(x,y)$ for the set of $\e$-optimal paths in $\cal{R}(x,y)$. The following is Proposition 3.2 of \cite{BLPR}.

\begin{prop} Let $p > p_c(2)$. There are positive constants $\al, c_1, c_2$ so that for all $x,y \in \Z^2$, we have
\begin{enumerate}
\item For any $t > \al |x-y|_2$, 
\begin{align}
\prob_p \Big( \exists \gamma \in \Gamma_0([x],[y]) \:: | \gamma | > t \Big) \leq c_1 \exp \Big(-c_2 t \Big) \,.
\end{align}
\item For all $\e >0$, once $|y-x|$ is sufficiently large depending on $\e$, 
\begin{align}
\prob_p\Big( \forall \gamma \in \Gamma_\e([x],[y]) \:: \Haus\big(\gamma, \poly(x,y) \big) > \e |y-x|_2 \Big) \leq c_1 \exp \Big(-c_2 \log^2( |y-x|_2) \Big) \,,
\end{align}
where $\poly(x,y)$ is the linear segment connecting $x$ and $y$.
\end{enumerate}
\label{prop:2_geom_con}
\end{prop}

{\large\section{\B{The variational problem}}\label{sec:variational}}

Having reintroduced $\beta_p$ in Section \ref{sec:norm}, we now discuss the variational problem \eqref{eq:2_variational_problem} specialized to $\tau = \beta_p$. In fact, we will use nothing about $\beta_p$ in this section other than the fact that it is a norm. We need two results in order to prove Theorem \ref{thm:main_shape} and Theorem \ref{thm:main_asym}: an existence result and a stability result. We write the functional defined in \eqref{eq:1_iso_problem_boundary} for $\tau = \beta_p$ as $\tension$, and for $R \in \cal{R}$, we refer to $\tension(\pa R)$ as the \emph{surface energy of $R$}. We also introduce the \emph{$\beta_p$-length} of a rectifiable curve $\lambda: [0,1] \to \R^2$:
\begin{align}
\len_{\beta_p}(\lambda) := \sup_{n \in \N} \sup_{t_1 < \dots < t_n \in [0,1]} \sum_{j=1}^{n} \beta_p( \lambda(t_j) - \lambda(t_{j-1}) )\,.
\end{align}
We find it necessary to consider not just the variational problem \eqref{eq:2_variational_problem}, but a family of related problems. For $\al \in [-1,1]$, define the following isoperimetric problem for sets $R \in \cal{R}$:
\begin{align}
\text{minimize: }\, \frac{\cal{I}_p(\pa R)}{ \Leb(R)}\,, \hspace{5mm} \text{subject to: }\, \Leb(R) \leq 2 +\al \,
\label{eq:3_var_problem_epsilon}
\end{align}
The minimal value for \eqref{eq:3_var_problem_epsilon} is 
\begin{align}
\vp_p^{(2+\al)} := \inf \left\{ \frac{\tension(\pa R)}{ \Leb(R) } \:: \Leb(R) \leq 2+ \al\,, R \in \cal{R} \right\} \,,
\label{eq:3_vp_epsilon}
\end{align} 
and the set of optimizers for \eqref{eq:3_var_problem_epsilon} is defined below as
\begin{align}
\cal{R}_p^{(2+\al)} := \left\{ R \in \cal{R} \:: \Leb(R) \leq 2 + \al \,, \frac{\tension(\pa R ) }{\Leb(R) } = \vp_p^{(2+\al)} \right\} \,.
\label{eq:3_optimizer_set_epsilon}
\end{align}
Thus, in our new notation, the constant $\vp_p$ introduced in \eqref{eq:1_optimal_conductance} is denoted $\vp_p^{(2)}$ in this section, and the collection of optimizers $\cal{R}_p$ introduced in Theorem \ref{thm:main_shape} is denoted $\cal{R}_p^{(2)}$.\\

\subsection{\small\textsf{Sets of finite perimeter}} We extend the problem \eqref{eq:3_var_problem_epsilon} to a larger class of sets, proving existence within this class and then recovering a representative in $\cal{R}$. For $E \subset [-1,1]^2$ be Borel, we define the \emph{perimeter} of $E$, denoted $\per(\pa E)$, as
\begin{align}
\per(\pa E) := \sup\left( \int_E \text{div}(f) dx \:: f \in C_c^\infty(\R^2, \R^2)\,,\, |f|_2 \leq 1 \right)\,,
\end{align}
and say that $E$ is a \emph{set of finite perimeter} if $\per(\pa E) < \infty$. Let $\cal{C}$ denote the collection of all sets of finite perimeter (after Caccioppoli) contained in $[-1,1]^2$. Given $E \in \cal{C}$, we define the \emph{$\beta_p$-perimeter} of $E$ similarly:
\begin{align}
\per_{\beta_p}(\pa E) := \sup\left( \int_E \text{div}(f) dx \:: f \in C_c^\infty(\R^2, \R^2)\,,\, \beta_p^*(f) \leq 1 \right)\,,
\end{align}
where $\beta_p^*$ is the dual norm to $\beta_p$. Finally, we define the surface energy of $E \in \cal{C}$ as:
\begin{align}
\tension(\pa E) := \sup\left( \int_{E} \text{div}(f) dx \:: f \in C_c^\infty((-1,1)^2, \R^2)\,,\, \beta_p^*(f) \leq 1 \right)\,.
\label{eq:3_tension_general}
\end{align}

\begin{rmk} Each $R \in \cal{R}$ is an element of $\cal{C}$, and the surface energy of $R$ defined in \eqref{eq:1_iso_problem_boundary} agrees with the surface energy of $E$, defined in \eqref{eq:3_tension_general}. This enables us to extend the variational problem \eqref{eq:3_var_problem_epsilon} to sets of finite perimeter, and given $E \in \cal{C}$, we call $\tension(\pa E) / \Leb(E)$ the \emph{conductance} of $E$, which is consistent with the terminology in the introduction.
\end{rmk}

We introduce the optimal value and set of optimizers corresponding to the variational problem over this wider class of sets. Define
\begin{align}
\psi_p^{(2+\al)} := \inf \left\{ \frac{\tension(\pa E)}{ \Leb(E) } \:: \Leb(E) \leq 2+ \al\,, E \in \cal{C} \right\} \,,
\label{eq:3_vp_epsilon_general}
\end{align} 
with the convention that zero divided by zero is infinity. Also define
\begin{align}
\cal{C}_p^{(2+\al)} := \left\{ E \in \cal{C} \:: \Leb(E) \leq 2 + \al \,, \frac{\tension(\pa E ) }{\Leb(E) } = \psi_p^{(2+\al)} \right\} \,.
\label{eq:3_optimizer_set_epsilon_general}
\end{align}
Lower semicontinuity is a fundamental feature of the perimeter and surface energy functionals (see for instance Section 14.2 of \cite{stflour}). 

\begin{lem} Let $E_k \in \cal{C}$ be a sequence converging in $L^1$-sense to $E$. Then 
\begin{enumerate}
\item $\per(\pa E) \leq \liminf_{k \to \infty} \per(\pa E_k)\,,$
\item $\per_{\beta_p}(\pa E) \leq \liminf_{k \to \infty} \per_{\beta_p}(\pa E_k)\,,$
\item $\tension(\pa E) \leq \liminf_{k \to \infty} \tension(\pa E_k)$,
\end{enumerate}
so that if $\per(\pa E_k)$ is uniformly bounded in $k$, we have $E \in \cal{C}$.
\label{lem:3_lsc} 
\end{lem}

We now introduce some terminology in order to state a result which linking the classes $\cal{R}$ and $\cal{C}$.

\begin{defn} Given $E \subset [-1,1]^2$ Borel, define the \emph{upper density} of $E$ at $x \in \R^2$ as
\begin{align}
D^+(E,x) := \limsup_{r \to 0} \frac{ \Leb( E \cap B(x,r)) }{ \Leb(B(x,r)) }\,,
\end{align}
and define the \emph{essential boundary} of $E$ as
\begin{align}
\pa^*E := \Big\{ x \in \R^2 \:: D^+(E,x) > 0,\, D^+(\R^2 \setminus E, x) >0 \Big\}
\end{align}
\end{defn}

\begin{defn} Let $E \subset \R^2$ be a set of finite perimeter. Say $E$ is \emph{decomposable} if there is a partition of $E$ into $A, B \subset \R^2$ so that $\Leb(A)$ and $\Leb(B)$ are strictly positive and so that $\per(\pa E) = \per(\pa A) + \per(\pa B)$. Say that $E$ is \emph{indecomposable} if it is not decomposable. 
\end{defn}

Recall that given a Jordan curve $\lambda$, we defined the compact set $\hull(\lambda)$ in \eqref{eq:2_hull_def}. We write $\hull(\lambda)^\circ$ for the interior of this compact set.  The following result, originally due to Fleming and Federer, allows us to think of $\pa^*E$ for $E \in \cal{C}$ as a countable collection of rectifiable Jordan curves. The version we state is taken roughly verbatim from Corollary 1 of \cite{Ambrosio}.

\begin{prop} Let $E \subset \R^2$ be a set of finite perimeter. There is a unique decomposition of $\pa^* E$ into rectifiable Jordan curves $\{ \lambda_i^+,\, \lambda_j^- \:: i,j \in \N\}$ (modulo $\cal{H}^1$-null sets) so that
\begin{enumerate}
\item For $i \neq k \in \N$, $\hull(\lambda_i^+)^\circ$ and $\hull(\lambda_k^+)^\circ$ are either disjoint, or one is contained in the other. Likewise, for $i \neq k \in \N$, $\hull(\lambda_i^-)^\circ$ and $\hull(\lambda_k^-)^\circ$ are either disjoint, or one is contained in the other. Each $\hull(\lambda_j^-)^\circ$ is contained in one of the $\hull(\lambda_i^+)^\circ$. 
\item $\per(\pa E) = \sum_{i=1}^\infty \cal{H}^1(\lambda_i^+) + \sum_{j=1}^\infty \cal{H}^1(\lambda_j^-)$.
\item If $\hull(\lambda_i^+)^\circ \subset \hull(\lambda_j^+)^\circ$ for $i \neq j$, then for some $\lambda_k^-$, we have $\hull(\lambda_i^+)^\circ \subset \hull(\lambda_k^-)^\circ \subset \hull(\lambda_j^+)^\circ$. Likewise, if $\hull(\lambda_i^-)^\circ \subset \hull(\lambda_j^-)^\circ$ for $i \neq j$, there is some $\lambda_k^+$ with $\hull(\lambda_i^-)^\circ \subset \hull(\lambda_k^+)^\circ \subset \hull(\lambda_j^-)^\circ$.
\item For $i \in \N$, let $L_i = \{ j \:: \hull(\lambda_j^-)^\circ \subset \hull(\lambda_i^+)^\circ \}$, and set 
\begin{align}
Y_i = \hull(\lambda_i^+) \setminus \left( \bigcup_{j \in L_i} \hull(\lambda_j^-)^\circ \right) \,.
\end{align}
The sets $Y_i$ are indecomposable with $\cal{H}^1$-null intersection, and moreover $\bigcup_{j=1}^\infty Y_j$ is equivalent of $E$ modulo Lebesgue null sets. 
\end{enumerate}
\label{prop:3_ambro}
\end{prop}

Proposition \ref{prop:3_ambro} tells us that sets of finite perimeter are in some sense extensions of the class $\cal{R}$ to sets whose boundary consists of countably many Jordan arcs instead of finitely many. Thus, it is reasonable that the theory of such sets comes into play when discussing limits of sets in $\cal{R}$. 

\subsection{\small\textsf{Existence}} We now show that $\cal{R}_p^{(2+\al)}$ is non-empty for all $\al \in [-1,1]$ by first using standard arguments to show $\cal{C}_p^{(2+\al)}$ is non-empty, and then by recovering elements of $\cal{R}$ from sets in $\cal{C}_p^{(2+\al)}$. We begin by making several basic observations. 

The first observation implies optimal Jordan domains must have full volume. 

\begin{lem} Let $\al \in [-1,1]$. Let $R \in \cal{R}$ be such that $\Leb(R) < 2 +\al$ and such that $R = \hull(\lambda)$ for a rectifiable Jordan curve $\lambda \subset [-1,1]^2$. Then there is also $R' \in \cal{R}$ with $\Leb(R) = 2 + \al$ and $R' = \hull(\lambda')$ for a rectifiable Jordan curve $\lambda' \subset [-1,1]^2$ with 
\begin{align}
\frac{\tension(\pa R)}{\Leb(R)} > \frac{ \tension(\pa R') }{\Leb(R') } \,.
\end{align}
\label{lem:3_full_vol_Jordan}
\end{lem}

\begin{proof} Let $R \in \cal{R}$ be as above, and consider the open set $A = (-1,1)^2 \setminus R$. We consider three cases. \\

\emph{\B{Case I:}} In the first, each connected component $A'$ of $A$ is such that $\pa A'$ intersects the interior of at most two adjacent sides of $[-1,1]^2$ non-trivially. In this first case, we can easily shrink the connected components of $A$ to form a new open set of arbitrarily small volume, and whose surface energy is at most that of $A$. By complementation, we recover $R'$ with the desired properties. 
\begin{figure}[h]
\centering
\includegraphics[scale=.5]{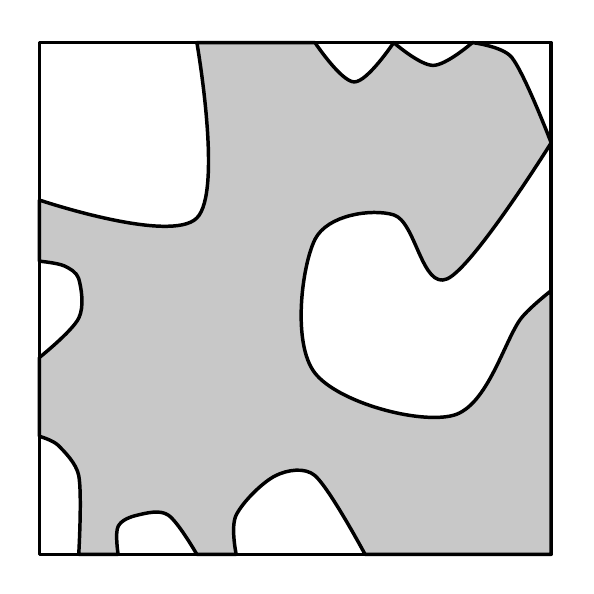} \includegraphics[scale=.5]{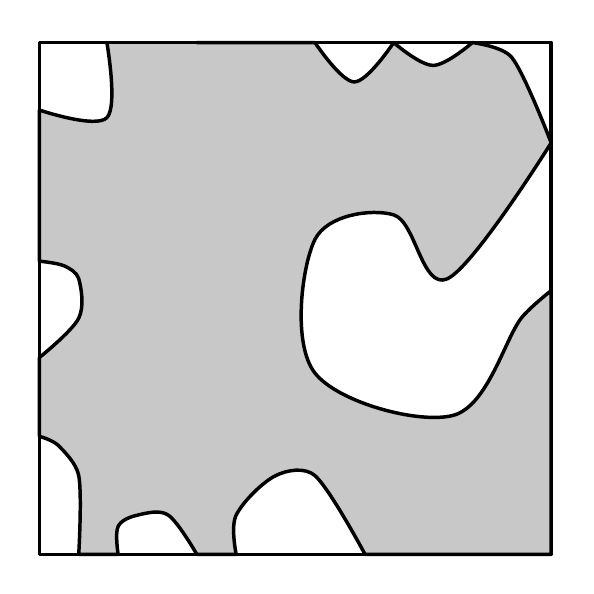}
\caption{On the left, the original set $R \in \cal{R}$ in grey. On the right, the set $R' \in \cal{R}$ obtained through the procedure described in \emph{\B{Case I}}.} 
\label{fig:full_vol_Jordan1}
\end{figure}

\emph{\B{Case II:}} In the second case, there is a connected component $A'$ of $A$ such that $\pa A'$ intersects the interior of exactly three sides of $[-1,1]^2$ non-trivially. Because $R$ is connected, $\pa A' \cap (-1,1)^2$ consists of a single arc which joins two opposing faces of the square, and this arc may be translated until it touches one of the other faces of the square. These translations naturally  yield sets of the desired form and of larger measure. If the measure of these sets surpasses $2+\al$ before the arc reaches the boundary, we are content. Otherwise, we have obtained a set which is handled by the previous case (upon performing the same procedure on at most one other arc, perhaps). 
\begin{figure}[h]
\centering
\includegraphics[scale=.5]{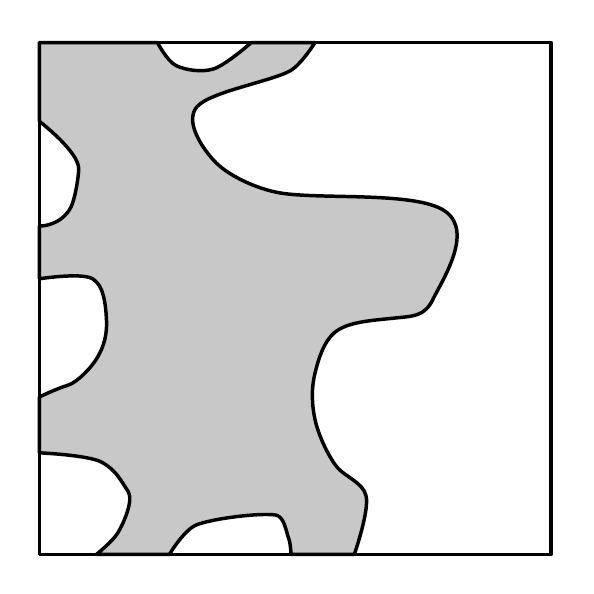} \includegraphics[scale=.5]{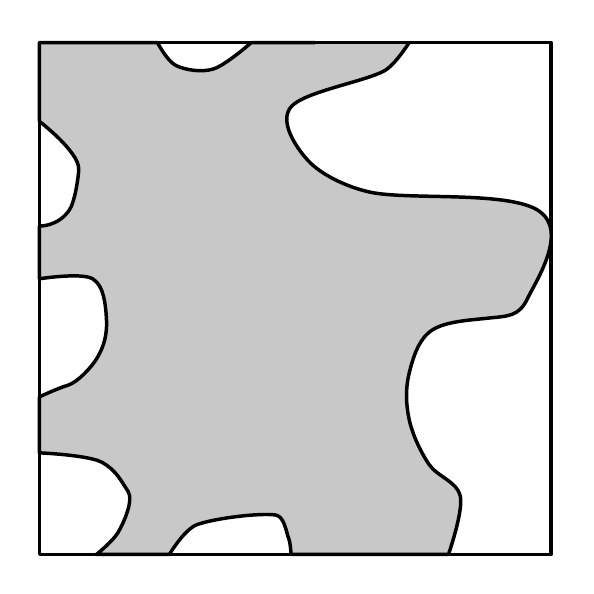}
\caption{On the left, the original $R \in \cal{R}$ in grey. On the right, $R'$ is obtained by ``sliding" one of the contours along the boundary of the box.} 
\label{fig:full_vol_Jordan2}
\end{figure}

\emph{\B{Case III:}} As $R$ must be connected, it is impossible for any connected component $A'$ of $A$ to have the property that $\pa A'$ intersects the interiors of two opposite sides of $[-1,1]^2$ non-trivially. Thus the last case to consider is that there is a connected component $A'$ of $A$ such that $\pa A'$ intersects the interior of all four sides of $[-1,1]^2$ non-trivially. In this case, $\pa R$ intersects the interiors of at most two adjacent sides of $[-1,1]^2$ non-trivially. We may then dilate $R$ about the corner it contains or the side it rests against until we either have a set of the desired measure or we have a set falling into one of the preceding cases. 
\begin{figure}[h]
\centering
\includegraphics[scale=.5]{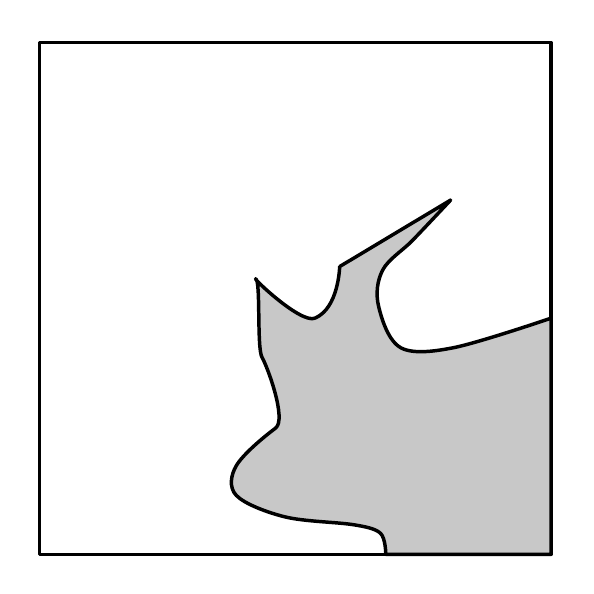} \includegraphics[scale=.5]{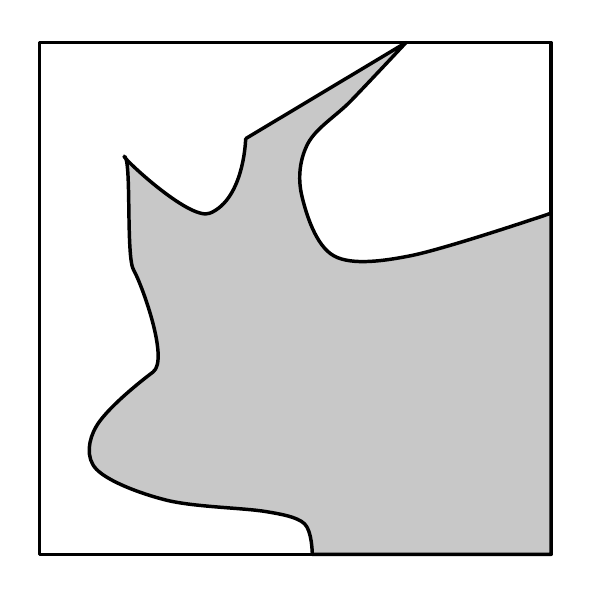}
\caption{On the left, $R \in \cal{R}$ is in grey. On the right, $R' \in \cal{R}$ is obtained by dilating $R$.} 
\label{fig:full_vol_Jordan3}
\end{figure}

This completes the proof. \end{proof}

Lemma \ref{lem:3_full_vol_Jordan} implies that optimal sets of finite perimeter must also have full volume.

\begin{lem} Let $\al \in [-1,1]$, and let $E \in \cal{C}$ with either $\Leb(E) < 2 + \al$, or $\Leb(E) \leq 2+\al$ and $E$ decomposable. There is $E' \in \cal{C}$ with $\Leb(E') = 2 + \al$ so that 
\begin{align}
\frac{\tension(\pa E)}{\Leb(E)} > \frac{ \tension(\pa E') }{\Leb(E') } \,.
\end{align}
\label{lem:3_J_reduction}
\end{lem}

\begin{proof} The case that $\Leb(E) \leq 2 +\al$ and $E$ is decomposable is an immediate corollary of the case $\Leb(E) < 2 +\al$, so we work in the latter. Thanks to the inequality $\tfrac{a+b}{c+d} \geq \min( \tfrac{a}{c} , \tfrac{b}{d})$, we lose no generality supposing $E$ is indecomposable. Using Proposition \ref{prop:3_ambro}, it follows that $E$ may be represented by rectifiable Jordan arcs $\lambda$ and $\{ \lambda_j \}_{j \geq 1}$ so that up to a Lebesgue-null set, $E = \hull(\lambda) \setminus \bigcup_{j\geq 1} \hull(\lambda_j)^\circ$.  As the curves $\lambda, \lambda_j$ have $\cal{H}^1$-null intersection, the sets $\hull(\lambda_j)^\circ$ are pairwise disjoint. Under the hypothesis that $\Leb(E) < 2+\al$, we may then shrink the curves $\lambda_j$ one by one to produce a set $E'$ having strictly smaller conductance. Thus, it suffices to consider sets $E$ of finite perimeter which may be represented by a single rectifiable Jordan curve $\lambda$, but this is handled entirely by Lemma \ref{lem:3_full_vol_Jordan}. \end{proof} 

We may now deduce that the collection of optimizers for  \eqref{eq:3_optimizer_set_epsilon} is non-empty within the class of sets of finite perimeter. 

\begin{lem} The set of optimizers $\cal{C}_p^{(2+\al)}$ for the variational problem \eqref{eq:3_optimizer_set_epsilon} is non-empty.
\label{lem:3_existence}
\end{lem}

\begin{proof} Let $E_k \in \cal{C}$ be a sequence of sets of finite perimeter such that 
\begin{align}
\frac{ \tension(\pa E_k) }{ \Leb(E_k)}  \to \psi_p^{(2+\al)} \,.
\end{align}
By Lemma \ref{lem:3_J_reduction}, we lose no generality supposing $\Leb(E_k) = 2 +\al$ for each $k$. As $\psi_p^{(2 +\al)}$ is clearly finite, the perimeters of the $E_k$ are uniformly bounded. We appeal to Rellich-Kondrachov and pass to a subsequence of the $E_k$ converging to some $E \subset [-1,1]^2$ in $L^1$-sense. By Lemma \ref{lem:3_lsc}, it follows that $E$ is a set of finite perimeter with $\Leb(E) = 2+\al$ and $\tension(\pa E) \leq \liminf_{k\to \infty} \tension(\pa E_k)$. Thus the conductance of $E$ is at most $\psi_p^{(2+\al)}$, which implies $E \in \cal{C}_p^{(2+\al)}$.
\end{proof}

We may now deduce that $\cal{R}_p^{(2+\al)}$ is non-empty for $\al \in [-1,1]$, among other things. The following is the main result of this subsection. 

\begin{coro} Let $\al \in [-1,1]$. 
\begin{enumerate}
\item If $E \in \cal{C}_p^{(2+ \al)}$, then $E$ is indecomposable and $\Leb(E) = 2 + \al$.
\item $E \in \cal{C}_p^{(2+ \al)}$ if and only if $E^c \in \cal{C}_p^{(2-\al)}$.
\item $\tfrac{2 +\al}{2-\al} \psi_p^{(2+\al)} = \psi_p^{(2-\al)}$. \label{eq:3_duality_third}
\item Each $E \in \cal{C}_p^{(2+\al)}$ is equivalent up to a Lebesgue-null set to some $R \in \cal{R}$. Thus, $\cal{R}_p^{(2+\al)}$ is non-empty and $\vp_p^{(2+\al)} = \psi_p^{(2 +\al)}$.
\item If $E \in \cal{C}_p^{(2+\al)}$, there are rectifiable Jordan curves $\lambda, \lambda' \subset [-1,1]^2$ so that up to Lebesgue-null sets, $E = \hull(\lambda)$ and $E^c = \hull(\lambda')$. Moreover, $\lambda \cap \lambda'$ is a simple rectifiable curve joining distinct points on $\pa [-1,1]^2$. \label{eq:3_duality_fifth}
\end{enumerate}
\label{coro:3_duality}
\end{coro}

\begin{proof} The first assertion is an immediate consequence of Lemma \ref{lem:3_J_reduction}. Because each $E \in \cal{C}_p^{(2+\al)}$ satisfies $\Leb(E) = 2+ \al$, and because $\tension(\pa E) = \tension(\pa E^c)$, the second and third assertions follow. Thus, whenever $E \in \cal{C}_p^{(2+\al)}$, both $E$ and $E^c$ are indecomposable. By Proposition \ref{prop:3_ambro}, either $E$ or $E^c$ is equivalent to $\hull(\lambda)$ for some rectifiable Jordan curve $\lambda \subset [-1,1]^2$, and the fourth assertion follows.

Turning our attention to the fifth assertion, it suffices to show that if $E \in \cal{C}_p^{(2+\al)}$ for $\al \in [-1,0]$, and if $E = \hull(\lambda)$ for a rectifiable Jordan curve $\lambda \subset [-1,1]^2$, then $\cal{H}^1( \lambda \cap \pa [-1,1]^2) > 0$. But this follows from the fact that if $\cal{H}^1( \lambda \cap \pa [-1,1]^2) =0$, the curve $\lambda$ at best can be the boundary of (a dilate of) the Wulff shape $W_p$ (this is the limit shape of \cite{BLPR} which is the unique solution, up to translation, of the unrestricted isoperimetric problem associated to the norm $\beta_p$). However, this shape is not optimal. For instance, a suitably dilated quarter-Wulff shape has strictly better conductance. \end{proof}

Let us include one last result to be used in the proof of Theorem \ref{thm:main_shape}, and which guarantees the non-degeneracy of the limit in Theorem \ref{thm:main_asym}.

\begin{lem} For each $\al \in [-1,1]$, we have $\vp_p^{(2+\al)} >0$. Moreover, for each $\al, \al' \in [-1,1]$ with $\al > \al'$, we have the strict monotonicity $\vp_p^{(2+\al')} > \vp_p^{(2+\al)}$. 
\label{lem:3_strict_mon}
\end{lem}

\begin{proof} Strict monotonicity follows from Lemma \ref{lem:3_J_reduction}. It suffices to show $\vp_p^{(3)}$ is positive; let us see how this follows from the fifth assertion of Corollary \ref{coro:3_duality}. Given $R \in \cal{R}_p^{(2+\al)}$, we have from \eqref{eq:3_duality_fifth} that $\pa R \cap (-1,1)^2$ is a simple rectifiable curve $\eta$ joining distinct points on the boundary of $\pa [-1,1]^2$. There are three short cases.\\

\emph{\B{Case I:}} We suppose the endpoints of $\eta$ lie on the same side of $\pa [-1,1]^2$. Thus, either $R$ or $R^c$ intersects at most one side of $[-1,1]^2$, and we let $A$ denote the set among $R$ and $R^c$ with this property. By reflecting $A$ about the side it borders, we produce a set $A'$ of twice the volume, with $\tension(\pa A') = 2 \tension( \eta) \equiv 2 \len_{\beta_p}(\eta)$. As $\Leb(A) \geq 1$, we use the standard Euclidean isoperimetric inequality to deduce
\begin{align}
\tension(\eta) \equiv \len_{\beta_p}(\eta) \geq \frac{c}{\sqrt{2}} \beta_p^{\min} \,,
\end{align}
where $c >0$ is some absolute constant, and where $\beta_p^{\min}$ is the minimum of $\beta_p$ over the unit circle.\\

\emph{\B{Case II:}} In the second case, we suppose the endpoints of $\eta$ lie on two adjacent sides of $\pa [-1,1]^2$. Either $R$ or $R^c$ intersects only these two sides of the square, and as before we let $A$ denote the set among $R$ and $R^c$ with this property. We proceed as before, except we now reflect twice, obtaining $A'$ with four times the volume of $A$, and with $\tension(\pa A') = 4 \tension( \eta) \equiv 4 \len_{\beta_p}(\eta)$. Thus,
\begin{align}
\tension(\eta) \equiv \len_{\beta_p}(\eta) \geq \frac{c}{2} \beta_p^{\min} \,,
\end{align}
with $c$ and $\beta_p^{\min}$ as above.\\

\emph{\B{Case III:}} In the final case, $\eta$ joins points on two opposing sides of $\pa [-1,1]^2$. In this case, it is clear that $\tension(\eta) \equiv \len_{\beta_p}(\eta) \geq 2\beta_p^{\min}$, where the two arises as the Euclidean distance between two opposing sides of the square.\\

In each case, we conclude that $\tension(\pa R) = \tension(\eta) >0$, which completes the proof. \end{proof}

\subsection{\small\textsf{Stability for connected sets}} Now that we have shown the set $\cal{R}_p^{(2+\al)}$ is non-empty, we show a stability result with respect to the $\Haus$-metric. First, some preliminary results. 

\begin{lem} Let $\al \in (-1,1)$. Suppose that $E_k \in \cal{C}$ are such that $\Leb(E_k) \leq 2+\al$ and the conductances of the $E_k$ tend to $\vp_p^{(2 +\al)}$. Then $\liminf_{k\to \infty} (E_k) > 0$, and if the $E_k \to E$ in $L^1$-sense, we have $E \in \cal{C}_p^{(2+\al)}$. 
\label{lem:3_conductances_L1}
\end{lem}

\begin{proof} Let $\al' \in (-1,1)$ be strictly less than $\al$. If $\Leb(E_k) \to 0$, we would have $\vp_p^{(2+\al')} \geq \vp_p^{(2+\al)}$, which contradicts Lemma \ref{lem:3_strict_mon}. Thus if the $E_k$ tend to $E \subset [-1,1]^2$ in $L^1$-sense, it follows that $\Leb(E) >0$. By Lemma \ref{lem:3_lsc}, we have
\begin{align}
\vp_p^{(2+\al)} = \liminf_{k\to \infty} \frac{\tension(E_k)}{\Leb(E_k) } \geq \frac{\tension(E)}{\Leb(E)} \,,
\end{align}
and thus $E \in \cal{C}_p^{(2+\al)}$.\end{proof}

For $E \in \cal{C}$ indecomposable, Proposition \ref{prop:3_ambro} tells us that $E$ is equivalent (up to a Lebesgue-null set) to $\hull(\lambda) \setminus \left( \bigcup_{j \geq 1} \hull(\lambda_j)^\circ \right)$ for $\lambda, \lambda_j \subset [-1,1]^2$ rectifiable Jordan curves. Given $E \in \cal{C}$ indecomposable, define $\wh{E} := \hull(\lambda)$, where $\lambda$ corresponds to $E$ as above. 

The next result tells us that if a sequence $E_k$ of indecomposable sets of finite perimeter tend to an optimal set, the size of the ``holes" in these sets must tend to zero. 

\begin{lem} Let $\al \in (-1,1)$. Let $E_k \in \cal{C}$ be indecomposable with $\Leb(E_k) \leq 2+ \al$ for all $k \geq 1$. Suppose that the $E_k$ tend to $E \in \cal{C}_p^{(2+\al)}$ in $L^1$-sense. Then as $k  \to \infty$, we have
\begin{align}
\Leb( \wh{E}_k \setminus E_k ) \to 0\,.
\end{align}
\label{lem:3_hole_fill_one}
\end{lem}

\begin{proof} Suppose not, and let $\al' \in (-1,1)$ be strictly larger than $\al$. We lose no generality supposing that $\Leb(\wh{E}_k \setminus E_k) \geq \e$ for all $k$. Moreover, by Lemma \ref{lem:3_conductances_L1}, we also lose no generality supposing that $\Leb(E_k) \geq 2 +\al - \e/2$ for all $k$ (using the fact at each $E \in \cal{C}_p^{(2+\al)}$ satisfies $\Leb(E) = 2 +\al$). 

Note that the $E_k^c$ also converge in $L^1$-sense to $E^c \in \cal{C}_p^{(2-\al)}$. The sets $E_k^c$ however are not indecomposable by hypothesis: let $A_k$ be the component of $E_k^c$ of smallest conductance, so that the conductance of $E_k^c$ serves as an upper bound for the conductance of $A_k$. But our hypotheses on the volumes of $\wh{E}_k$ and $E_k$ ensure that $\Leb(A_k) \leq 2-\al -\e/2$, which implies that $\vp_p^{(2-\al')} \leq \vp_p^{(2-\al)}$, contradicting Lemma \ref{lem:3_strict_mon}. \end{proof}

Heuristically, the above lemma allows us to replace a sequence of sets in $\cal{R}$ by Jordan domains. The next result tells us that a sequence of Jordan domains converging in the correct sense to an element of $\cal{C}_p^{(2+\al)}$ has a limit in $\cal{R}$. 

\begin{lem} Let $R_k \in \cal{R}$ be a sequence such that $R_k = \hull(\lambda_k)$ for rectifiable Jordan curves $\lambda_k \subset [-1,1]^2$, and suppose that the conductances of the $R_k$ tend to $\vp_p^{(2+\al)}$. Suppose also that $R_k \to K$ both in $L^1$-sense and in $\Haus$-sense, where $K \subset [-1,1]^2$ is compact and $K \in \cal{C}_p^{(2+\al)}$. Then $K \in \cal{R}_p^{(2+\al)}$. 
\label{lem:3_regularity_recovery}
\end{lem}

\begin{proof} In this proof, we carefully distinguish curves (continuous functions from $[0,1]$ into $[-1,1]^2$ taking the same value at $0$ and $1$) from their images. Given a curve $\lambda: [0,1] \to [-1,1]^2$, let $\image(\lambda)$ denote the image of $\lambda$. As $K \in \cal{C}_p^{(2+\al)}$, the perimeters of the $\pa R_k$ are uniformly bounded. By appealing to an arc length parametrization of each $\lambda_k$, we may assume each $\lambda_k$ is a Lipschitz function from $[0,1]$ to $[-1,1]^2$ with a uniform bound on the Lipschitz constant across all $k$. Invoking Arzela-Ascoli and passing to a subsequence, we find that the $\lambda_k$ tend uniformly to a rectifiable curve $\lambda$. 

By appealing to the definition of the hull of a curve (using winding number), we find that $\hull(\lambda_k) \to \hull(\lambda)$ in $\Haus$-sense, thus $K \equiv \hull(\lambda)$. Let $\wt{\lambda} : [0,1] \to (-1,1)^2$ be a reparametrization of $\lambda$ of constant speed, so that $K = \hull(\wt{\lambda})$ also. Suppose that $\wt{\lambda}$ is not a simple curve, and moreover suppose there is $x \in (-1,1)^2$ such that $| \wt{\lambda}^{-1} (x) | >1$. Let $s < t \in [0,1]$ be such that $\wt{\lambda}(s) = \wt{\lambda}(t)$. Let us write $\zeta_1 := \wt{\lambda}|_{[s,t)}$ and $\zeta_2 := \wt{\lambda}|_{[0,s] \cup (t,1]}$, so that both $\zeta_1$ and $\zeta_2$ are closed curves. 

As $K \in \cal{C}_p^{(2+\al)}$, the set $K$ must be indecomposable with indecomposable complement. It follows that $\hull(\wt{\lambda})^\circ$ is either $\hull(\zeta_1)^\circ$ or $\hull(\zeta_2)^\circ$. As $x \in (-1,1)^2$, we also have that $\tension(\wt{\lambda}) > \tension(\zeta_1)$ and $\tension(\wt{\lambda}) > \tension(\zeta_2)$. Without loss of generality then, we have 
\begin{align}
\frac{\tension( \pa K) }{ \Leb(K) } \leq \frac{ \tension(\zeta_1)}{\Leb(K)} < \frac{ \tension(\wt{\lambda})}{\Leb(K)} \leq \vp_p^{(2+\al)} \,,
\end{align}
where the right-most inequality follows from lower semicontinuity of the surface energy Lemma~\ref{lem:3_lsc} (and the hypothesis that the conductances of the $R_k$ tend to the optimal value). This is a contradiction. Thus, if $|\wt{\lambda}^{-1}(x)| >1$, it must be that $x \in \pa [-1,1]^2$, and there exists a Jordan curve $\lambda' \subset [-1,1]^2$ such that $\hull(\lambda') = \hull(\wt{\lambda}) = K$. We conclude that $K \in \cal{R}_p^{(2+\al)}$. 
\end{proof}

Lemma \ref{lem:3_regularity_recovery} essentially allows us to recover some regularity of a suitable limit of Jordan domains. We now use this to show that the collections $\cal{R}_p^{(2)}$ and $\cal{R}_p^{(2+\al)}$ are close when $\al$ is small. 

\begin{lem} Let $\al \in (0,1]$. As $\al \to0$, we have $\Haus( \cal{R}_p^{(2+\al)}, \cal{R}_p^{(2)}) \to 0$. 
\label{lem:3_opt_close}
\end{lem}

\begin{proof} Let $\al_k \in (0,1]$ be a sequence tending to zero as $k \to \infty$. Let $R_k \in \cal{R}_p^{(2 +\al_k)}$. By Corollary \ref{coro:3_duality} \eqref{eq:3_duality_fifth}, there are rectifiable Jordan curves $\lambda_k \subset [-1,1]^2$ with $R_k = \hull(\lambda_k)$. By Corollary \ref{coro:3_duality} \eqref{eq:3_duality_third}, the conductances of the $R_k$ tend to $\vp_p^{(2)}$. 

The non-empty compact subsets of $[-1,1]^2$ form a compact metric space when equipped with the $\Haus$-metric. We pass to a subsequence (twice, using Rellich-Kondrachov) so that we lose no generality supposing $R_k \to K$ in $\Haus$-sense and in $L^1$-sense, where $K \subset [-1,1]^2$ is compact. As $\Leb(E_k) \to 2$ as $k \to \infty$, the lower semicontinuity of the surface energy (Lemma \ref{lem:3_lsc}) implies $K \in \cal{C}_p^{(2+\al)}$. We apply Lemma \ref{lem:3_regularity_recovery} to conclude that $K \in \cal{R}_p^{(2+\al)}$ to complete the proof. \end{proof}

The following is the first of two stability results, and is a precursor to the main result in this subsection. 

\begin{prop} Let $\al \in (-1,1)$ and let $\e >0$. There is $\delta = \delta(\al, \e) >0$ so that whenever $R \in \cal{R}$ is connected with $\Leb(R) \leq 2 + \al$ and $\Haus(R, \cal{R}_p^{(2+\al)} ) > \e$, we have
\begin{align}
\frac{\tension(\pa R)}{ \Leb(R) } \geq \vp_p^{(2+\al)} + \delta
\end{align}
\label{prop:3_qualitative}
\end{prop}

\begin{proof} Suppose not. Then there is a sequence $R_k \in \cal{R}$ of connected sets with $\Leb(R_k) \leq 2 +\al$, and with $\Haus(R_k, \cal{R}_p^{(2+\al)}) > \e$ and 
\begin{align}
\frac{ \tension(\pa R_k) }{ \Leb(R_k) } \to \vp_p^{(2+\al)} \,.
\end{align}
Suppose first that for each $k$, $R_k = \hull(\lambda_k)$, where $\lambda_k \subset [-1,1]^2$ is a rectifiable Jordan curve. By Rellich-Kondrachov, and by the compactness of the set of non-empty compact subsets of $[-1,1]^2$ in the metric $\Haus$, we lose no generality (by passing to a subsequence) supposing that $R_k \to K \subset [-1,1]^2$ compact, where the convergence takes place both in $L^1$-sense and in $\Haus$-sense. By Lemma~\ref{lem:3_conductances_L1}, it follows that $K \in \cal{C}_p^{(2+\al)}$, and by Lemma \ref{lem:3_regularity_recovery}, it then follows that $K \in \cal{R}_p^{(2+\al)}$, which is a contradiction. 

Let us then suppose that none of the $R_k$ are of the form $\hull(\lambda_k)$ for a sequence of rectifiable Jordan curves $\lambda_k \subset [-1,1]^2$, so that for each $k$, we have $\wh{R}_k \neq R_k$. We appeal to the same compactness argument as above, and suppose that the $R_k$ tend to $K \subset [-1,1]^2$ compact both in $L^1$-sense and in $\Haus$-sense. As before, Lemma \ref{lem:3_conductances_L1} tells us $K \in \cal{C}_p^{(2+\al)}$. We then use Lemma \ref{lem:3_hole_fill_one} to deduce that $\Leb(\wh{R}_k \setminus R_k) \to 0$. 

As the conductances of the $R_k$ tend to $\vp_p^{(2+\al)}$, and as $\vp_p^{(2+\al +\e)} \to \vp_p^{(2+\al)}$ as $\e \to 0$, it follows that the diameter of any connected component of $\wh{R}_k \setminus R_k$ must also tend to zero. Thus, as $k \to \infty$, we have that $\Haus(\wh{R}_k , R_k) \to 0$, and we may then realize $K \in \cal{C}_p^{(2+\al)}$ as the $L^1$- and $\Haus$-limit of the $\wh{R}_k$. As each $\wh{R}_k$ is the hull of a rectifiable Jordan curve, we may now use Lemma \ref{lem:3_regularity_recovery} to deduce that $K \in \cal{R}_p^{(2+\al)}$, which is again a contradiction. \end{proof}

Our second stability result upgrades Proposition \ref{prop:3_qualitative}, removing the $\al$ dependence of the constant $\delta$. It is the main result of this subsection and is instrumental to the proof of Theorem \ref{thm:main_shape}.

\begin{coro} Let $\al \in (0,1]$ and let $\e >0$. There is $\delta = \delta(\e) >0$ so that whenever $R \in \cal{R}$ is connected with $\Leb(R) \leq 2 + \al$ and $\Haus(R, \cal{R}_p^{(2+\al)} ) > \e$, we have
\begin{align}
\frac{\tension(\pa R)}{ \Leb(R) } \geq \vp_p^{(2+\al)} + \delta
\end{align}
\label{coro:3_stability_final}
\end{coro}

\begin{proof} Let $\e >0$ and let $\al_k$ be a sequence in $(0,1]$ tending to zero as $k \to \infty$. Let $\til{\delta}(\al_k,\e)$ be the supremum of all $\delta >0$ for which Proposition \ref{prop:3_qualitative} is valid for the parameters $\al_k$ and $\e$. Then, for each $k$, there are connected sets $R_k \in \cal{R}$ with $\Leb(R_k) \leq 2 +\al_k$ so that $\Haus(R_k, \cal{R}_p^{(2+\al_k)}) \geq \e$ and 
\begin{align}
\frac{ \tension(\pa R_k) }{ \Leb(R_k) } \leq \vp_p^{(2 + \al_k)} + 2 \til{\delta}(\al_k, \e) \,.
\end{align}
Suppose for the sake of contradiction that $\til{\delta}(\al_k ,\e) \to 0$ as $k \to \infty$. Then the conductances of the $R_k$ tend to $\vp_p^{(2)}$. Passing to a subsequence, we may assume that $R_k \to K$ compact with $K \in \cal{C}_p^{(2+\al)}$, where the convergence takes place both in $L^1$-sense and in $\Haus$-sense. If each $R_k$ is the hull of a rectifiable Jordan curve, we may invoke Lemma \ref{lem:3_regularity_recovery} to deduce that $K \in \cal{R}_p^{(2+\al)}$. If not, we may proceed as in the proof of Proposition \ref{prop:3_qualitative}, replacing each $R_k$ by $\wh{R}_k$ to deduce the same result.

Thus, the $R_k$ get arbitrarily close in $\Haus$-sense to $\cal{R}_p^{(2)}$, so that for all $k$ sufficiently large, $\Haus(R_k, \cal{R}_p^{(2)} ) \leq \e/4$. Thanks to Lemma \ref{lem:3_opt_close}, we may also find $k$ sufficiently large so that $\Haus( \cal{R}_p^{(2+ \al_k)} , \cal{R}_p^{(2)} ) < \e/4$. This contradicts the fact that $\Haus(E_k, \cal{R}_p^{(2 + \al_k) } ) > \e$\end{proof}

{\large\section{\B{Continuous to discrete: upper bounds}}\label{sec:upper}}

In this section, we show that given $R \in \cal{R}$ with $\Leb(R) \leq 2$, there are high probability upper bounds on $n\Chee$ in terms of the conductance of $R$. We show first this for polygons and then use approximation to pass to more general sets.\\

\subsection{\small\textsf{From simple polygons to discrete sets}}

A \emph{convex polygon} in $\R^2$ is a compact subset of $\R^2$ having non-empty interior which may be written as the intersection of finitely many closed half-spaces. A \emph{polygon} is any subset of $\R^2$ which may be written as a finite union of convex polygons.

Recall (from the statement of Proposition \ref{prop:2_geom_con}) that given $x,y \in \R^2$, we use $\poly(x,y)$ to denote the linear segment joining $x$ and $y$. Given a sequence of points $x_1, \dots, x_m$, we define
\begin{align}
\poly(x_1, \dots, x_m) := \poly(x_1, x_2) * \dots * \poly(x_{m-1}, x_m) \,,
\label{eq:4_poly_concat}
\end{align}
where $``*"$ denotes concatenation of these curves. A \emph{polygonal curve} is any curve of the form \eqref{eq:4_poly_concat} for some $x_1, \dots, x_m \in \R^2$ and some $m \in \N$ (we return to being vague about the parametrization). Polygons may be defined from polygonal curves in a natural way; we say a polygon is \emph{simple} if it may be written as the hull of a simple polygonal circuit. The first proposition of this section associates a discrete set to any simple polygon in a convenient way.

\begin{rmk} In this section and the next we will be somewhat cavalier with notation. In particular, for $R \in \cal{R}$, the dilated set $n R$ is not in general contained in $[-1,1]^2$. The surface energy of $nR$, denoted $\tension(n \pa R)$ is defined to be $n \tension(\pa R)$. We employ a similar convention for curves. 

\end{rmk}

\begin{prop} Let $p > p_c(2)$ and let $\e >0$. Let $P \subset [-1,1]^2$ be a simple non-degenerate polygon.  There are positive constants $c_1(p,P,\e)$ and $c_2(p,P,\e)$ so that for all $n \geq 1$, with probability at least $1 - c_1 \exp(-c_2 \log^2n)$, there is a rectifiable circuit $\lambda \equiv \lambda(P) \subset [-1,1]^2$ so that  

\begin{enumerate}
\item $\Haus(n\pa P,  n\lambda) \leq \e n$,
\item $\tension(n\pa P) \geq (1 -\e) | \pa^n [ \hull(n\lambda) \cap \cluster ]|$. 
\end{enumerate}
\label{prop:4_simple_poly_carve}
\end{prop}

\begin{proof} \B{\emph{Step I: (Aggregation of high probability events)}} Let $x_1, \dots, x_m$ be the corners of $nP$, so that
\begin{align}
nP = \hull( \poly(x_1, \dots, x_{m})) \,,
\end{align}
where $x_m \equiv x_1$, and where the circuit $\poly(x_1, \dots, x_m)$ is oriented counter-clockwise. Let $\cal{E}_{1}$ be the high probability event from Lemma \ref{lem:2_close_cluster} that for each $i \in \{1, \dots , m\}$, we have $| [x_i] - x_i|_2 \leq \log^2n$. Say $x_i$ is an \emph{interior point} if $x_i \in (-n,n)^2$, and that it is a \emph{boundary point} otherwise. For $n$ sufficiently large, the Euclidean ball $B_{2\log^2n}(x_i)$ is contained in $(-n,n)^2$ for each interior point $x_i$. For such $n$ and within $\cal{E}_1$, we have $[x_i] \in (-n,n)^2$ for each interior $x_i$. 

For $\delta >0$, define the high probability event $\cal{E}_{2}(\delta)$ via
\begin{align}
\cal{E}_{2}(\delta) := \bigcap_{i=1}^{m-1} \Big\{ \exists \gamma \in \Gamma_\delta(x_i, x_{i+1}) \:: {\rm{d}}_H \Big( \gamma, \poly(x_i, x_{i+1} ) \Big) \leq \delta | x_{i+1} - x_i| _2 \Big\} \,,
\end{align}
so that $\cal{E}_{2}(\delta)^c$ is subject to the bounds in Proposition \ref{prop:2_geom_con}. Additionally, define
\begin{align}
\cal{E}_{3}(\delta) := \bigcap_{i=1}^{m-1}\left\{ \left| \frac{ b( [x_i], [x_{i+1}]) }{\beta_p(x_{i+1}-x_i) } -1 \right| > \delta \right\} \,,
\end{align}
so that $\cal{E}_{3}(\delta)^c$ is subject to the bounds in Theorem \ref{thm:2_meas_con}. For the remainder of the proof, work within the intersection $\cal{E}_1\cap \cal{E}_2(\delta) \cap \cal{E}_3(\delta)$. \\

\emph{\B{Step II: (Constructing $\lambda$)}} Select $\gamma_i \in \Gamma_\delta(x_i, x_{i+1})$ with ${\rm{d}}_H( \gamma_i, \poly(x_i, x_{i+1}) ) < \delta |x_{i+1} - x_i |_2$ for each $i \in \{1,\dots, m-1\}$. Each $\gamma_i$ may be identified with an interface $\pa_i$ via the correspondence in Proposition \ref{prop:2_interface_path}.

\begin{figure}[h]
\centering
\includegraphics[scale=1.3]{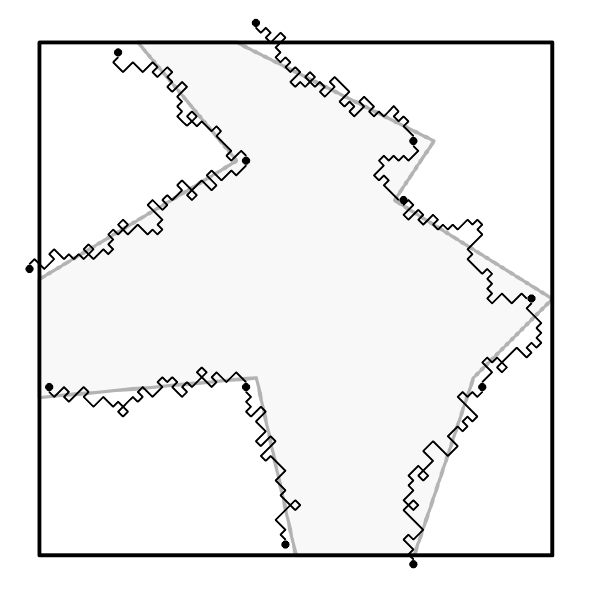}
\caption{ The polygon $nP$ is in grey. The black dots are the $[x_i]$, and the contours joining these dots are the $\pa_i \equiv \lambda_i$ corresponding to the interior segments $\poly(x_i,x_{i+1})$.} 
\label{fig:ice_carve}
\end{figure}

A linear segment $\poly(x_i, x_{i+1})$ is an \emph{interior segment} if at least one of $x_i$ or $x_{i+1}$ is an interior point, and otherwise it is a \emph{boundary segment}. If $\poly(x_i, x_{i+1})$ is a boundary segment, set $\lambda_i := \poly(x_i, x_{i+1})$, otherwise, via ``corner-rounding" (see Remark \ref{rmk:2_corner_round}), regard $\pa_i$ as a simple curve and set $\lambda_i := \pa_i$. If the endpoint of $\lambda_i$ is not equal to the starting point of $\lambda_{i+1}$, let $\wt{\lambda}_i$ be the linear segment joining these two points. If $\lambda_i$ ends at the starting point of $\lambda_{i+1}$, let $\wt{\lambda}_i$ be the degenerate linear segment at this endpoint. Define the circuit $n\lambda$ as the concatenation of these curves in the proper order:
\begin{align}
n\lambda := \lambda_1 * \wt{\lambda}_1 * \lambda_2 * \wt{\lambda}_2 * \dots * \lambda_m * \wt{\lambda}_m \,, 
\end{align}
and write $H_n$ for  $\hull(n\lambda) \cap \giant$. Let $E_i$ be the set of all edges of $\Z^2$ contained in the Euclidean ball $B_{2\log^2n}(x_i)$, so that by construction of $H_n$, 
\begin{align}
| \pa^n H_n | \leq \sum_{\substack{i \::\: {\tiny\textsf{poly}}(x_i, x_{i+1})\\ \text{is interior}}} |\frak{b}( \gamma_i)| + \sum_{i=1}^m | E_i | \,.
\label{eq:4_chee_upper_0}
\end{align}

\emph{\B{Step III: (Controlling $|\pa^n H_n|$)}} We build off \eqref{eq:4_chee_upper_0} and use that each $\gamma_i$ is $\delta$-optimal (see \eqref{eq:e_optimal}), 
\begin{align}
| \pa^n H_n | &\leq  \sum_{\substack{i \::\: {\tiny\textsf{poly}}(x_i, x_{i+1})\\ \text{is interior}}} \Big( b([x_i], [x_{i+1}] ) + \delta|x_{i+1} - x_i|_2 \Big) + \sum_{i=1}^m | E_i | \,, \\
&\leq \left( \sum_{\substack{i \::\: {\tiny\textsf{poly}}(x_i, x_{i+1})\\ \text{is interior}}}b([x_i], [x_{i+1}]) \right) + 2mn\delta + C \log^4n \,,
\end{align}
for some absolute positive constant $C$. As we are within $\cal{E}_2(\delta)$, for $n$ sufficiently large we have
\begin{align}
| \pa^n H_n | &\leq \left( \sum_{\substack{i \::\: {\tiny\textsf{poly}}(x_i, x_{i+1})\\ \text{is interior}}}(\beta_p( x_{i+1} - x_i ) + n\delta) \right)  + 4mn\delta \,, \\
&\leq \tension(n\pa P) + 8mn\delta \,.
\label{eq:4_chee_upper_6}
\end{align}

\emph{\B{Step IV: (Wrapping up)}} Given $\e >0$, we may choose $\delta$ sufficiently small depending on $P$ and $\e$ so that from \eqref{eq:4_chee_upper_6}, we have
\begin{align}
\tension(n\pa P) \geq (1-\e) |\pa^n H_n| \,.
\end{align}
Finally, the construction of $\lambda$ from the $\gamma_i$ ensures that 
\begin{align}
\Haus(nP, n\lambda) \leq 2\delta \max_{i=1}^{m-1} | x_{i+1} - x_i |_2 \,,
\end{align}
and we take $\delta$ smaller if necessary to complete the proof.
\end{proof}

\subsection{\small\textsf{Upper bounds on $n\Chee$ using connected polygons}}

We now use the output of Proposition \ref{prop:4_simple_poly_carve} to construct a discrete approximate to more general connected polygons. We also relate the volume of the discrete approximate to the volume of this polygon. 

\begin{prop} Let $p > p_c(2)$ and let $\e >0$. Let $P \subset [-1,1]^2$ be a connected polygon whose boundary consists of finitely many disjoint simple polygonal circuits.  There are positive constants $c_1(p,P,\e)$ and $c_2(p,P,\e)$ so that for all $n \geq 1$, with probability at least $1 - c_1 \exp(-c_2 \log^2n)$, there is subgraph $H_n \equiv H_n(P) \subset \giant$ so that
\begin{enumerate}
\item $| \theta_p \Leb(nP) - |H_n| | \leq \e \Leb(nP)$,
\item $\tension(n\pa P) \geq (1 -\e) | \pa^n H_n|$. 
\end{enumerate}
\label{prop:4_poly_to_disc}
\end{prop}

\begin{proof} \emph{\B{Step I: (Using circuits to identify $H_n$)}} Using the hypotheses on $P$, identify disjoint, simple polygonal circuits $\rho, \rho_1, \dots, \rho_m$ so that $\pa P = \rho \sqcup \bigsqcup_{i=1}^m \rho_m$, and so that 
\begin{align}
P = \hull(\rho)\setminus \left( \bigsqcup_{i=1}^{m} \hull(\rho_i)^\circ \right)\,,
\end{align}
where the $\hull(\rho_i)$ are pairwise disjoint, and where each is a simple polygon. For $\delta >0$, work within the high probability events given by Proposition \ref{prop:4_simple_poly_carve} that there are circuits $\lambda, \lambda_1, \dots, \lambda_m\subset [-1,1]^2$ so that 
\begin{enumerate}
\item $\Haus(n\rho_i, n\lambda_i) \leq \delta n$ for each $i$, and $\Haus(n\rho, n\lambda) \leq \delta n$.
\item $\tension(n \rho_i) \geq (1-\delta) | \pa^n [ \hull( n\lambda_i) \cap \cluster ] |$ for each $i$, and $\tension(n \rho) \geq (1-\delta) | \pa^n [ \hull( n\lambda) \cap \cluster ] |$ \label{eq:4_poly_to_disc_1}
\end{enumerate}
Define the set
\begin{align}
R := \hull(\lambda)\setminus \left( \bigsqcup_{i=1}^{m} \hull(\lambda_i)^\circ \right)\,,
\end{align}
and let $H_n := nR \cap \giant$. By \eqref{eq:4_poly_to_disc_1}, the graph $H_n$ has the second desired property:
\begin{align}
\tension(n\pa P) \geq (1-\delta) |\pa^n H_n| \,.
\end{align}

\emph{\B{Step II: (Controlling the volume of $H_n$ from above)}} We control the volume of $H_n$ by appealing to Proposition \ref{prop:4_gandolfi}. Let $k \in \N$ and let $\sfS_k$ denote the set of half-open dyadic squares at the scale $k$ which are contained in $[-1,1]^2$; these are translates of $[-2^{-k}, 2^{-k})^2$. For $\delta' >0$ and $S \in \sfS_k$, define the event
\begin{align}
\cal{E}_S(\delta') := \left\{ \frac{ |\cluster \cap nS| }{ {\Leb}(nS) } \in \Big( (1-\delta')\theta_p, (1+\delta') \theta_p\Big) \right\} \,,
\end{align}
and let $\cal{E}_{\rm{vol}} (\delta')$ be the intersection of $\cal{E}_S(\delta')$ over all $S \in \sfS_k$. From now on, work within the event $\cal{E}_{\rm{vol}}(\delta')$. Let $N_{2\delta}$ be the closed $2\delta$-neighborhood (with respect to Euclidean distance) of $\pa P$. Let $\sfS_k^-$ be the squares of $\sfS_k$ contained in $P \setminus N_{2\delta}$, and let $\sfS_k^+$ be the squares of $\sfS_k$ having non-empty intersection with $P \cup N_{2\delta}$. Here we assume $\delta$ is small enough and $k$ is large enough for $\sfS_k^-$ to be non-empty. Thanks to the construction of $H_n$, we have 
\begin{align}
| H_n| &\leq \sum_{S \in \sfS_k^+} | nS \cap \cluster | + Cn \,,
\label{eq:4_chee_upper_1}
\end{align}
where $C$ is some absolute constant, and the term $Cn$ directly above accounts for the vertices of $\Z^2$ in $\pa [-n,n]^2$, which we must be mindful of as the squares $S \in \sfS_k$ are half-open. Choose $k$ large enough depending on $\delta'$ and $P$ so that 
\begin{align}
(1-\delta') \Leb( P) \leq \sum_{S \in \sfS_k^-} \Leb(S) \leq \sum_{S \in \sfS_k^+} \Leb(S) \leq (1+\delta') \Leb( P)\,.
\label{eq:4_chee_upper_2}
\end{align}
For $n$ sufficiently large, it follows from \eqref{eq:4_chee_upper_1}, \eqref{eq:4_chee_upper_2}, the fact that we are working within $\cal{E}_{\rm{vol}}(\delta')$ that
\begin{align}
|H_n| &\leq (1 + 2\delta')^2 \theta_p \Leb(n P) \,. 
\label{eq:4_chee_upper_3}
\end{align}

\emph{\B{Step III: (Controlling the volume of $H_n$ from below)}} Work within the following high probability event from Proposition \ref{prop:4_benj_mo} for the remainder of the proof:
\begin{align}
\Big\{ \cluster \cap [-n + \log^2n, n- \log^2n ] = \giant \cap [-n + \log^2n, n- \log^2n] \Big\}  \,.
\end{align}
We appeal to the construction of $H_n$ and the disjointness of the squares in $\sfS_k$, taking $n$ sufficiently large to obtain the second line below:
\begin{align}
| H_n | &\geq \left( \sum_{S \in \sfS_k^-} | \cluster \cap nQ_j| \right) - | \cluster \cap [-n,n]^2 \setminus \giant|  \,.\\
&\geq (1-2\delta') \sum_{S \in \sfS_k^-} \theta_p \Leb(nS) \,,\\
&\geq (1-2 \delta')(1-\delta') \theta_p\Leb(nP) \,.
\label{eq:4_chee_upper_5}
\end{align}
where the last line follows from \eqref{eq:4_chee_upper_2}. We choose $\delta, \delta'$ sufficiently small to complete the proof. \end{proof}

We now use Proposition \ref{prop:4_poly_to_disc} to obtain high probability upper bounds on $\Chee$ in terms of the conductance of a connected, non-degenerate polygon which is not too large. 

\begin{coro} Let $p > p_c(2)$ and let $\e >0$. Let $P \subset [-1,1]^2$ be a connected polygon with $\Leb(P) < 2$, and whose boundary is a finite disjoint union of simple polygonal curves. There are positive constants $c_1(p,P,\e)$ and $c_2(p,P,\e)$ so that for all $n\geq1$, with probability at least $1- c_1 \exp(-c_2 \log^2n)$, 
\begin{align}
n\Chee \leq (1 + \e) \frac{ \tension(\pa P)}{\theta_p \Leb(P)} \,.
\end{align}
\label{coro:4_chee_upper_poly}
\end{coro}

\begin{proof} Define $\e' := 2 - \Leb(P)$ and let $\delta >0$. By combining Proposition \ref{prop:4_gandolfi} with Proposition~\ref{prop:4_benj_mo}, we obtain positive constants $c_1(p,\delta)$ and $c_2(p,\delta)$ so that the probability of the event 
\begin{align}
\left\{\frac{| \giant|}{(2n)^2} \in \Big( (1-\delta)\theta_p, (1+\delta)\theta_p \Big) \right\}\,
\label{eq:4_chee_upper_poly_1}
\end{align}
is at least $1 - c_1 \exp(-c_2 \log^2n)$. Work within this high probability event, and additionally work within the high probability event from Proposition \ref{prop:4_poly_to_disc} that there is $H_n \subset \giant$ satisfying
\begin{enumerate}
\item $| \theta_p \Leb(nP) - |H_n| | \leq \delta \Leb(nP)$,
\item $\tension(n\pa P) \geq (1 -\delta) | \pa^n H_n|$. 
\end{enumerate}
Thus, $|H_n| \leq (\theta_p + \delta) (2-\e') n^2$. Using \eqref{eq:4_chee_upper_poly_1} and choosing $\delta$ small enough depending on $\e'$ so that $2(1-\delta)\theta_p \geq (\theta_p + \delta) (2-\e')$, we find $|H_n| \leq |\giant|/2$, and conclude that with high probability,
\begin{align}
\Chee \leq \frac{ | \pa^n H_n| }{|H_n|} \leq \frac{ \tfrac{1}{1-\delta} \tension(nP) }{ (\theta_p -\delta) \Leb(nP) } \,,
\end{align}
which completes the proof, taking $\delta$ smaller if necessary. \end{proof}

\subsection{\small\textsf{The optimal upper bound on $n\Chee$}} We now exhibit a high probability upper bound on $n \Chee$ using the optimal conductance of $\vp_p$ defined in \eqref{eq:1_optimal_conductance}. We introduce results which allow us to approximate rectifiable Jordan curves by simple polygonal circuits. The following consolidates Lemma 4.3 and Lemma 4.4 of \cite{BLPR}.

\begin{prop} Let $\lambda$ be a rectifiable curve in $\R^2$ starting at $x$ and ending at $y$. Let $\e >0$. There is a simple polygonal curve $\rho$ starting at $x$ and ending at $y$ such that (1) and (2) hold:
\begin{enumerate}
\item $\Haus ( \lambda, \rho ) \leq \e \,,$
\item $\len_{\beta_p} (\lambda) +\e \geq \len_{\beta_p}(\rho) \,.$
\end{enumerate}
Furthermore, if $\lambda$ is a closed curve (i.e. $x = y$), $\rho$ can additionally be taken to satisfy (3):
\begin{enumerate}
\setcounter{enumi}{2}
\item $\Leb ( \hull(\lambda)\, \Delta\, \hull(\rho) ) \leq \e \,.$
\end{enumerate}
\label{prop:4_rect_approx}
\end{prop}

\begin{rmk} We remark that, in Proposition \ref{prop:4_rect_approx}, if the curve $\lambda$ is contained in $[-1,1]^2$, one can easily arrange that the polygonal approximate $\rho$ is also contained in $[-1,1]^2$.
\label{rmk:4_in_bounds}
\end{rmk}

The following is a nearly immediate consequence Proposition \ref{prop:4_rect_approx}, so we omit the proof.

\begin{coro} Let $\lambda \subset [-1,1]^2$ be a rectifiable Jordan curve such that $\lambda = \lambda_1 * \lambda_2$, where $\lambda_1$ and $\lambda_2$ are simple curves with $\lambda_1 \subset \pa [-1,1]^2$, and such that every point on the curve $\lambda_2$ except the endpoints lies in $(-1,1)^2$. Let $\e >0$. There is a simple polygonal circuit $\rho \subset [-1,1]^2$ so that 
\begin{enumerate}
\item $\Haus ( \lambda, \rho ) \leq \e \,,$ \label{eq:4_rect_approx_2_1}
\item $\cal{I}_{p} (\lambda) +\e \geq \cal{I}_{p}(\rho) \,,$
\item $\Leb ( \hull(\lambda)\, \Delta\, \hull(\rho) ) \leq \e \,.$\label{eq:4_rect_approx_2_3}
\end{enumerate}
\label{coro:4_rect_approx_2}
\end{coro}

\begin{rmk} If instead of a decomposition of $\lambda$ into two curves as in Corollary \ref{coro:4_rect_approx_2}, we express $\lambda$ as a concatenation of finitely many curves, each having the properties of $\lambda_1$ or $\lambda_2$, the conclusion of Corollary \ref{coro:4_rect_approx_2} still holds. That is, for such $\lambda$, we may find a polygonal circuit $\rho$ for which \eqref{eq:4_rect_approx_2_1} -- \eqref{eq:4_rect_approx_2_3} hold.
\label{rmk:4_finite_concat}
\end{rmk}

We are now equipped to prove Theorem \ref{thm:4_chee_upper_final}, which is the main theorem of the section.

\begin{thm} There are positive constants $c_1(p,\e)$ and $c_2(p,\e)$ so that for all $n\geq 1$, with probability at least $1- c_1 \exp(-c_2 \log^2n)$,
\begin{align}
n\Chee \leq (1+\e) \vp_p \,,
\end{align}
where $\vp_p$ is defined in \eqref{eq:1_optimal_conductance}.
\label{thm:4_chee_upper_final}
\end{thm}

\begin{proof} Let $R \in \cal{R}_p$. By Corollary \ref{coro:3_duality}, we lose no generality taking $R = \hull(\lambda)$, with $\lambda$ as in the statement of Corollary \ref{coro:4_rect_approx_2}. For $\delta >0$, there is a simple polygonal circuit $\rho \subset [-1,1]^2$ so that
\begin{enumerate}
\item $\Haus ( \lambda, \rho ) \leq \delta \,,$
\item $\cal{I}_{p} (\lambda) +\delta \geq \cal{I}_{p}(\rho)\,,$
\item $\Leb ( \hull(\lambda)\, \Delta\, \hull(\rho) ) \leq \delta\,.$
\end{enumerate}
As $R = \hull(\lambda)$ has positive measure, there is $s >0$ and a square of side-length $S$ which is contained in the interior $R$. For $\delta$ sufficiently small, $S$ is also contained in the interior of $\hull(\rho)$. Let $P_s := \hull(\rho) \setminus S^\circ$, and observe that $P_s$ is a connected polygon satisfying the hypotheses of Corollary~\ref{coro:4_chee_upper_poly}, as well as
\begin{enumerate}
\setcounter{enumi}{3}
\item $2-\delta-s^2 \leq \Leb(P_s) \leq 2 +\delta - s^2$\,, \label{eq:4_chee_upper_final_1}
\item $\tension(\pa R) + \delta + 4s \beta_p^{\max} \geq \tension(\pa P_s)$\,,\label{eq:4_chee_upper_final_2}
\end{enumerate}
where  $\beta_p^{\max}$ is the maximum of $\beta_p$ over the unit circle. By taking $\delta$ smaller if necessary so that $s^2 > 2\delta$, we find $\Leb(P_s) < 2$.  Thus, by Corollary \ref{coro:4_chee_upper_poly}, with high probability
\begin{align}
n \Chee &\leq (1+ \delta) \frac{ \tension(\pa P_s)}{ \theta_p \Leb(P_s) }  \,,\\
&\leq (1 + \delta) \frac{ \tension(\pa R) + \delta +4\beta_p^{\max}s }{\theta_p ( \Leb(R) - \delta - s^2) } \,,
\end{align}
where we have used \eqref{eq:4_chee_upper_final_1} and \eqref{eq:4_chee_upper_final_2}. The proof is complete upon adjusting $\delta$ and $s$. \end{proof}

{\large\section{\B{Discrete to continuous objects: lower bounds}}\label{sec:disc_to_cts}}

We construct tools which allow us to pass from a subgraph of $\giant$ to a connected polygon of comparable conductance. By Lemma \ref{lem:5_circuit_decomp}, the boundary of a subgraph of $\giant$ may be thought of as a finite collection of open right-most circuits. Our first goal is then to construct an approximating polygonal curve to any open right-most path.\\

\subsection{\small\textsf{Extracting polygonal curves from right-most paths}} Our first result enables us to pass from open right-most paths of sufficient length to polygonal curves.

\begin{lem} Let $p > p_c(2)$ and let $\e >0$. There are positive constants $c_1(p,\e)$ and $c_2(p,\e)$ so that for all $n\geq 1$, with probability at least $1 - c_1 \exp (-c_2  \log^2n)$, whenever $\gamma \subset [-n,n]^2$ is an open right-most path with $| \gamma | \geq n^{1/32}$, there is a simple polygonal curve $\rho = \rho(\gamma) \subset [-1,1]^2$ with
\begin{enumerate}
\item $\Haus(\gamma, n\rho) \leq n^{1/64}$\,,
\item $|\frak{b}(\gamma)| \geq (1-\e) \len_{\beta_p}(n\rho)$\,.
\end{enumerate} 
\label{lem:5_poly_curve}
\end{lem}

\begin{proof} For $x, y \in [-n,n]^2 \cap \Z^2$ and $\e >0$, let $\cal{E}_{x,y}$ be the event
\begin{align}
\cal{E}_{x,y} := \left\{ \left| \frac{ b( [x], [y]) }{\beta_p(y-x) } -1 \right| \leq \e \right\}\,.
\end{align}
Let $\cal{E}$ be the intersetion of all $\cal{E}_{x,y}$ over pairs $x,y \in [-n,n]^2 \cap \Z^2$ satisfying $|x-y|_2 \geq n^{1 / 1024}$, and work within $\cal{E}$ for the remainder of the proof. By Theorem~\ref{thm:2_meas_con} and a union bound, there are positive constants $c_1(p,\e)$ and $c_2(p,\e)$ so that 
\begin{align}
\prob_p( \cal{E}^c) \leq  c_1 \exp \left( -c_2  \log^2n \right) \,.
\end{align}

\emph{\B{Step I: (Constructing a polygonal curve)}} Consider an open right-most path $\gamma \subset [-n,n]^2$ with $|\gamma| \geq n^{1/32}$ and express $\gamma$ as an alternating sequence of vertices and edges:
\begin{align}
\gamma = (x_0, e_1, x_1 , \dots, e_m, x_m) \,.
\label{eq:5_poly_curve_1}
\end{align}
Define a subsequence of the vertices $x_i$ as follows: let $\ell$ be the largest positive integer such that $(\ell-1) \lceil n^{1 /256} \rceil \leq m$, and for $k \in \{0, \dots, \ell-1\}$, set 
\begin{align}
y_k := x_{k \lceil n^{1/256}\rceil} \,
\end{align}
and set $y_\ell := x_m$. Because $\gamma$ is right-most, no vertex $x_j$ in \eqref{eq:5_poly_curve_1} appears more than four times. Thus for $n$ sufficiently large, $|y_{k+1}- y_k |_2 \geq n^{1 /1024}$ for all $k \in \{0, \dots, \ell-2\}$. Let $\rho' \subset [-1,1]^2$ be the polygonal curve defined by
\begin{align}
n\rho' := \poly(y_0, y_1) * \poly(y_1,y_2) * \dots * \poly(y_{\ell-1},y_\ell) \,. 
\end{align}
We check that $\rho'$ has the desired properties and finish the proof by perturbing $\rho'$ to a simple polygonal curve for which these properties still hold. \\

\emph{\B{Step II: (Controlling the $\beta_p$-length of $\rho'$ from above)}} As $m = | \gamma| \geq n^{1/32}$, it follows that $\ell \geq \tfrac{1}{2} n^{(1/32)-(1/256)}$. Because $|y_{k+1}- y_k |_2 \geq n^{1 /1024}$ for $k \in \{0, \dots, \ell-2\}$, we deduce
\begin{align}
\len_{\beta_p}( n\rho') \geq c(p) n^{29/1024} \,
\label{eq:5_poly_curve_0}
\end{align}
for some positive constant $c(p)$. Because we are within $\cal{E}$, and because each $\gamma_k$ is open and right-most,
\begin{align}
|\frak{b}(\gamma)| &\geq \sum_{k=0}^{\ell-1} b(y_k, y_{k+1}) \,, \\
&\geq  (1-\e) \len_{\beta_p}( n\rho') - \beta_p( y_\ell - y_{\ell-1}) \,.
\end{align}
As $\beta_p( y_\ell -y_{\ell-1}) \leq C(p) ^{1/256}$ for some positive constant $C(p)$, by taking $n$ sufficiently large and using \eqref{eq:5_poly_curve_0}, we find
\begin{align}
|\frak{b}(\gamma)| &\geq (1-2\e) \len_{\beta_p}(n\rho') \,.
\label{eq:5_poly_curve_b}
\end{align}

\emph{\B{Step III: ($\Haus$-closeness of $n\rho'$ and $\gamma$)}} For $k \in \{0, \dots, \ell-1\}$, let $\gamma_k$ be the subpath of $\gamma$ starting at $y_k$ and ending at $y_{k+1}$. Observe that every vertex in $\gamma_k$ has $\ell^\infty$-distance at most $2\lceil n^{1/256}\rceil$ from the starting point $y_k$. Regarding $\gamma_k$ as a curve, we see $\Haus(\gamma_k, y_k) \leq 2\lceil n^{1 /256}\rceil$. Likewise, $\Haus(\poly(y_k, y_{k+1}), y_k ) \leq 2\lceil n^{1 /256}\rceil$, so for each $k \in \{0, \dots, \ell-1\}$, we have 
\begin{align}
\Haus( \gamma_k, \poly(y_k, y_{k+1}) ) \leq 4\lceil n^{1 /256}\rceil \,,
\end{align}
and hence, for $n$ taken sufficiently large, we have the following desirable bound:
\begin{align}
\Haus( \gamma, n \rho' ) &\leq n^{1 /128}\,.
\label{eq:5_poly_curve_a}
\end{align}

\emph{\B{Step IV: (Perturbation)}} It remains to perturb $\rho'$ to a simple polygonal curve. For $\delta >0$, use Proposition \ref{prop:4_rect_approx} (and Remark \ref{rmk:4_in_bounds}) to obtain a simple polygonal curve $\rho \subset [-1,1]^2$ so that $\Haus(\rho,\rho') \leq \delta$ and so that $\len_{\beta_p} (\rho') + \delta  \geq \len_{\beta_p}(\rho)$. Using \eqref{eq:5_poly_curve_a} and \eqref{eq:5_poly_curve_b}, we find
\begin{enumerate}
\item $\Haus( \gamma, n\rho) \leq n^{1/128} + n\delta$,
\item $|\frak{b}(\gamma)| \geq (1-2\e) ( \len_{\beta_p}(n\rho) - n\delta)$,
\end{enumerate}
and the proof is complete upon setting $\delta = \min( n^{(1/128)-1}, \e \len_{\beta_p} (\rho') )$, adjusting $\e$ and taking $n$ larger if necessary.  \end{proof}

Our second result allows us to pass from right-most circuits of sufficient length to polygonal circuits. Note that the boundary of $[-1,1]^2$ now comes into play: we obtain control on the surface energy of the polygonal circuit (as opposed to simply the $\beta_p$-length) in terms of the $\giant$-length of the right-most circuit (as opposed to the $\cluster$-length). 

\begin{lem} Let $p > p_c(2)$ and let $\e >0$. There are positive constants $c_1(p,\e)$ and $c_2(p,\e)$ so that with probability at least $1 - c_1 \exp (-c_2 \log^2n)$, whenever $\gamma \subset [-n,n]^2$ is an open right-most circuit with $| \gamma | \geq n^{1/4}$, there is a simple polygonal circuit $\rho = \rho(\gamma) \subset [-1,1]^2$ with
\begin{enumerate}
\item $\Haus(\gamma, n\rho) \leq n^{1/16}$\,, 
\item $|\frak{b}^n(\gamma)| \geq (1-\e) \tension(n\rho)$\,.
\end{enumerate} 
Moreover, if $\gamma \subset (-n,n)^2$, we may replace (2) above with
\begin{enumerate}
\setcounter{enumi}{2}
\item $|\frak{b}^n(\gamma)| \geq (1-\e) \len_{\beta_p}(n\rho)$\,.
\end{enumerate} 
\label{lem:5_poly_curve_2}
\end{lem}

\begin{proof} Let $\gamma \subset [-n,n]^2$ be an open right-most circuit with $|\gamma | \geq n^{1/4}$, and express $\gamma$ as an alternating sequence of vertices and edges
\begin{align}
\gamma =( x_0, e_1, x_1, e_2, x_2, \dots, e_m, x_m) \,,
\label{eq:5_1_concat}
\end{align}
where $x_0 = x_m$. \\

\emph{\B{Step I: (Decomposition of $\gamma$)}} Say that $x_i$ is a \emph{boundary vertex} if $x \in \pa [-n,n]^2$ and that $x_i$ is an \emph{interior vertex} otherwise. If no $x_i$ in $\gamma$ is a boundary vertex, our analysis is simplified, so we postpone dealing with this case. As $\gamma$ is a circuit, we lose no generality supposing $x_0$ is a boundary vertex. Let $\wt{x}_0, \dots, \wt{x}_\ell$ enumerate the boundary vertices of $\gamma$ ordered in terms of increasing index in \eqref{eq:5_1_concat}. For $j \in \{1, \dots, \ell\}$, let $\gamma_j$ be the subpath of $\gamma$ starting at $\wt{x}_{j-1}$ and ending at $\wt{x}_j$. Each $\gamma_j$ is right-most and has the property that only the endpoints of $\gamma_j$ are boundary vertices.

Say $\gamma_j$ is \emph{long} if $| \gamma_j | \geq n^{1/32}$, and that it is \emph{short} otherwise. For each $\gamma_j$, let $\gamma_j'$ denote the unique self-avoiding path of edges contained in $\pa[-n,n]^2$ whose starting and ending points are those of $\gamma_j$. \\

\emph{\B{Step II: (Polygonal approximation)}} Let $\e >0$ and work within the high probability event from Lemma~\ref{lem:5_poly_curve} for this parameter. For each long $\gamma_j$, there is then a simple polygonal curve $\rho_j \subset [-1,1]^2$ satisfying 
\begin{enumerate}
\item $\Haus(\gamma_j, n\rho_j) \leq n^{1/64}$, \label{eq:5_poly_curve_2_close}
\item $ | \frak{b}(\gamma_j)| \geq (1-\e) \len_{\beta_p}(n\rho_j)$.\label{eq:5_poly_curve_2_energy}
\end{enumerate}
If $\gamma_j$ is short, the path $\gamma_j'$ may be regarded as a polygonal curve $n\rho_j \subset \pa[-n,n]^2$ joining $\wt{x}_{j-1}$ with $\wt{x}_j$. Thus, each $\gamma_j$ gives rise to a simple polygonal curve $\rho_j \subset [-1,1]^2$ in one of two ways, according to $|\gamma_j|$. Let $\rho'$ be the concatenation of the $\rho_j$ in the proper order:
\begin{align}
\rho' := \rho_1 *\dots * \rho_\ell\,,
\end{align}
so that $\rho'$ is a polygonal circuit. 

We claim $\rho'$ has the desired properties; we first check $\Haus$-closeness of $n\rho'$ and $\gamma$. If $\gamma_j$ is short, any vertex in $\gamma_j$ has an $\ell^\infty$-distance of at most $2n^{1/32}$ to $\wt{x}_j$, and likewise any vertex in $\gamma_j'$ has an $\ell^\infty$-distance of at most $2n^{1/32}$ to $\wt{x}_j$. It follows that $\Haus(\gamma_j, n\rho_j) \leq 4n^{1/32}$ when $\gamma_j$ is short. In the case that $\gamma_j$ is long, \eqref{eq:5_poly_curve_2_close} above provides even better control, and we conclude
\begin{align}
\Haus( \gamma, n\rho') \leq 4n^{1/32} + n^{1/64} \,.
\label{eq:5_poly_curve_2_Haus}
\end{align}

We now turn to controlling $\tension(n\rho')$. Using the decomposition $\gamma = \gamma_1 * \dots * \gamma_\ell$ and the construction of $\rho'$,
\begin{align}
| \frak{b}^n(\gamma)| &\geq \sum_{j \::\: \gamma_j \text{ long}} | \frak{b}(\gamma_j)| \,,\\
&\geq (1-\e) \sum_{j \::\: \gamma_j \text{ long}} \len_{\beta_p} (n\rho_j) \,, \\
&\geq (1-\e) \tension(n\rho') \,,
\label{eq:5_poly_curve_2_tension}
\end{align}
where we have used \eqref{eq:5_poly_curve_2_energy} to obtain the second line directly above.\\

\emph{\B{Step III: (Perturbation)}} It remains to perturb $\rho'$ to a simple polygonal circuit. Let $\delta >0$, and apply Corollary \ref{coro:4_rect_approx_2} (and Remark \ref{rmk:4_finite_concat}) to $\rho'$ with this $\delta$, so that by \eqref{eq:5_poly_curve_2_Haus} we have
\begin{align}
\Haus(\gamma, n\rho) \leq 4n^{1/32} + n^{1/64} + \delta n\,,
\end{align}
and by \eqref{eq:5_poly_curve_2_tension} we have 
\begin{align}
| \frak{b}^n(\gamma)| \geq (1-\e) ( \tension(n\rho) - \delta n).
\end{align}
The proof is complete upon setting $\delta = \min( n^{(1/32) -1}, \e \tension(\rho'))$, adjusting $\e$ and taking $n$ larger if necessary. In the case that $\gamma$ contains no boundary vertices, we split $\gamma$ into a concatenation of two long right-most paths and proceed as above. \end{proof}

\subsection{\small\textsf{Interlude: optimizers are of order $n^2$}} In the arguments to come, it will be important to know that with high probability, each Cheeger optimizer has size on the order of $n^2$. First, we present a self-contained argument that $\Chee$ is at most a constant times $n^{-1}$ with high probability. This follows from results mentioned in the introduction, but the proof given here is short enough to include. 

\begin{prop} Let $p > p_c(2)$. There are positive constants $c(p), c_1(p), c_2(p) >0$ so that with probability at least $1 -c_1 \exp(-c_2 \log^2n)$, we have
\begin{align}
\Chee \leq cn^{-1} 
\end{align}
\label{prop:A_surface_order}
\end{prop}

\begin{proof} We use the previous two results to provide a high-probability lower bound on $|\giant|$. Fix $\delta >0$. Using Proposition \ref{prop:4_gandolfi} and Proposition \ref{prop:4_benj_mo}, we find that with probability at least $1 - c_1 \exp(-c_2\log^2n)$, 
\begin{align}
| \giant | &\geq | \cluster \cap [-n,n)^2 | - 4n \log^2n \,,\\
&\geq (\theta_p - \delta) (2n)^2 -4n \log^2n \,,\\
&\geq (\theta_p - 2\delta)(2n)^2\,, \label{eq:surface_order_one}
\end{align}
where we have taken $n$ sufficiently large to obtain the last line. Define $H_n := [-n/8, n/8)^2 \cap \giant$. Within the above events, we have $[-n/8,n/8)^2 \cap \giant = [-n/8, n/8)^2 \cap \cluster$, and thus we may also work within the high probability event that $| H_n| \in ( (\theta_p -\delta) (n/4)^2, (\theta_p +\delta) (n/4)^2 )$. Thus for $\delta$ chosen well, $|H_n| \leq |\giant| /2$. As $| \pa^n H_n|$ is at most a constant times $n$, we have shown that with high probability, $\Chee \leq cn^{-1}$ for some $c >0$.\end{proof}

We now deduce that with high probability, each Cheeger optimizer is large.

\begin{prop} Let $p > p_c(2)$. There are positive constants $c_1(p), c_2(p), \al(p)$ so that with probability at least $1 -c_1\exp(-c_2\log^2n)$, we have
\begin{align}
\min_{G_n \in \cal{G}_n} |G_n| \geq \al n^2 \,.
\end{align}
\label{prop:5_vol_order_opt}
\end{prop}

\begin{proof} We make two assumptions:
\begin{enumerate}
\item $G_n$ is connected. \label{eq:prop_vol_order_opt_ass_1}
\item $|G_n| \leq |\giant| /2 - n^{1/8}$ \label{eq:prop_vol_order_opt_ass_2}
\end{enumerate}
Use Lemma \ref{lem:5_circuit_decomp} and the fact that $G_n$ is connected to identify a right-most circuit $\gamma$ as in the statement of Lemma \ref{lem:5_circuit_decomp}. We now follow \emph{\B{Step I}} in the proof of Lemma \ref{lem:5_poly_curve_2} and write $\gamma$ as an alternating sequence of vertices and edges:
\begin{align}
\gamma =( x_0, e_1, x_1, e_2, x_2, \dots, e_m, x_m) \,,
\label{eq:vol_order_opt_1}
\end{align}
where $x_0 = x_m$. We say $x_i$ is a \emph{boundary vertex} if $x_i \in \pa [-n,n]^2$ and that it is an \emph{interior vertex} otherwise. We split the remainder of the proof into two cases.\\

\emph{\B{Case I:}} In the first case, we suppose $\gamma$ contains no boundary vertices, so that $\pa^\infty G_n = \pa^n G_n$. Thanks to Proposition~\ref{prop:A_BBHK}, the following event occurs with high probability:
\begin{align}
\left\{ \Lambda \subset \giant, \Lambda \text{ is connected}, | \Lambda| \geq n^{1/2} \implies |\pa^\infty \Lambda | \geq \wt{\al} | \Lambda|^{1/2} \right\} \,.
\end{align}
Work within this event, and also the high probability event from Proposition \ref{prop:A_surface_order} that $\Chee \leq cn^{-1}$. As $\giant$ is connected, it follows that $|\pa^n G_n| \geq 1$ for each Cheeger optimizer. Thus, within the high probability events in which we work, it follows that $|G_n| \geq c^{-1} n$, and that within this first case,
\begin{align}
|\pa^n G_n | = |\pa^\infty G_n| \geq \wt{\al} |G_n|^{1/2} \,,
\end{align}
so that $|G_n| \geq ( \wt{\al}  / c)^2 n^2$, which is desirable.\\

\emph{\B{Case II:}} In the second case, we suppose that $\gamma$ contains at least one boundary vertex, and we continue to follow \emph{\B{Step I}} in the proof of Lemma \ref{lem:5_poly_curve_2}. Without loss of generality, $x_0$ is then a boundary vertex and we let $\wt{x}_0, \dots, \wt{x}_\ell$ enumerate the boundary vertices of $\gamma$ in terms of increasing order in \eqref{eq:vol_order_opt_1}. For $j \in \{1, \dots, \ell\}$, we let $\gamma_j$ be the subpath of $\gamma$ which begins at $x_{j-1}$ and ends at $x_j$. As before, we note that each $\gamma_j$ is right-most and that only the endpoints of $\gamma_j$ are boundary vertices. We say that $\gamma_j$ is \emph{long} if $|\gamma_j| \geq n^{1/32}$ and that $\gamma_j$ is \emph{short} otherwise.  

We claim that no $\gamma_j$ can be short. To see this, let $\wt{\gamma}_j$ be the right-most path defined by the sequence of edges, each contained in $\pa [-n,n]^2$, and which begin at $\wt{x}_j$ and end at $\wt{x}_{j-1}$. Let $\pa_j$ be the counter-clockwise interface which corresponds to $\gamma_j * \wt{\gamma}_j$, and observe that 
\begin{align}
| \hull(\pa_j) \cap \giant | &\leq \Leb( \hull(\pa_j) ) + c | \gamma_j * \wt{\gamma}_j|  \,, \\
&\leq c \len( \pa_j)^2 + c | \gamma_j* \wt{\gamma}_j| \,,\\
&\leq c n^{1/16} < n^{1/8} \,.
\end{align}
Here, $c$ is an absolute constant which is allowed to change from line to line, and we have used the standard Euclidean isoperimetric inequality to obtain the second line. The third line follows from the assumption that $\gamma_j$ is short and by taking $n$ large. Writing $G_n' := G_n \cup [\hull(\pa_j) \cap \giant]$, and using \eqref{eq:prop_vol_order_opt_ass_2}, we have that $|G_n'| \leq |\giant| /2$ and that the conductance of $G_n'$ is strictly smaller than that of $G_n$. This is a contradiction, so our claim that no $\gamma_j$ can be short holds.

By Proposition \ref{prop:2_length_comparison}, it is a high-probability event that $| \frak{b}^n(\gamma_j) | \geq \al |\gamma_j|$. Thus, writing $\pa$ for the interface corresponding to $\gamma$, it follows that 
\begin{align}
| \pa^n G_n | &\geq | \frak{b}^n(\gamma) | \geq c \cal{H}^1\big( \pa \cap (-n,n)^2\big) \,,\\
&\geq c \Leb\big( \hull(\pa) \cap [-n,n]^2 \big)^{1/2}\\
&\geq c | G_n |^{1/2} \,,
\end{align}
where we've used the isoperimetric inequality to obtain the second line, and where the constant $c >0$ changes from line to line. 

This handles the second case, and it remains to address our assumptions \eqref{eq:prop_vol_order_opt_ass_1} and \eqref{eq:prop_vol_order_opt_ass_2}. If $|G_n| \geq |\giant|/2 - n^{1/8}$, we use \eqref{eq:surface_order_one} from the proof of Proposition \ref{prop:A_surface_order} and take $n$ large to see that $|G_n| \geq cn^{2}$ with high probability in this case. Finally, any $G_n$ is a disjoint union of connected Cheeger optimizers, so the lower bounds on the connected Cheeger optimizers suffice. \end{proof}

\subsection{\small\textsf{Approximating discrete sets via polygons}} Now that we have tools for converting right-most circuits to polygonal circuits, we use the decomposition given by Lemma \ref{lem:5_circuit_decomp} to pass from subgraphs of $\giant$ to connected polygons. In order to relate the conductances of these objects, we require a mild isoperimetric assumption on the subgraph of $\giant$. 

Recall that $\cal{U}_n$ denotes the collection of connected subgraphs of $\giant$ which inherit their graph structure from $\giant$. Given a decomposition of $U \in \cal{U}_n$ as in Lemma \ref{lem:5_circuit_decomp}, define 
\begin{align}
\dper(U) := |\gamma| + \sum_{j=1}^m | \gamma_j|,
\label{eq:5_dper_def}
\end{align}
which may be thought of as the ``full" perimeter of $U$. We also define
\begin{align}
\vol(U) := \hull(\pa) \setminus \left( \bigsqcup_{j=1}^m \hull(\pa_j) \right) \,,
\label{eq:5_vol_U_def}
\end{align}
where $\pa$ and the $\pa_j$ are the interfaces corresponding to the right-most circuits $\gamma, \gamma_j$.

\begin{defn} Say that $U \in \cal{U}_n$ is \emph{well-proportioned} if
\begin{align}
\dper(U) \leq \Leb( \vol(U) )^{2/3} \,.
\end{align}
\end{defn}

The following coarse-graining result says that with high probability, each $U \in \cal{U}_n$ is $\Haus$-close to $\vol(U)$. Moreover, if $U \in \cal{U}_n$ is well-proportioned and sufficiently large, we may deduce $U$ has ``typical" density within $\vol(U)$. This second statement is Lemma 5.3 of \cite{BLPR} rephrased, and we essentially follow the proof of this lemma to deduce Lemma \ref{lem:5_volume_for_good} below.

\begin{lem} Let $p >p_c(2)$ and let $\e >0$. There are positive constants $c_1(p,\e)$ and $c_2(p,\e)$ so that with probability at least $1- c_1 \exp(-c_2 \log^2n)$,
\begin{align}
\Haus(U, \vol(U)) \leq \log^4n\,.
\end{align}
Moreover, whenever $U \in \cal{U}_n$ satisfies
\begin{enumerate}
\item $U$ is well-proportioned,
\item $\Leb(\vol(U) ) \geq \log^{20}n$,
\end{enumerate}
we have
\begin{align}
\left| \frac{ |U|}{\Leb(\vol(U))} - \theta_p \right| < \e \,.
\end{align}
\label{lem:5_volume_for_good}
\end{lem}

\begin{proof} Let $\e >0$, and define $r: = \lfloor \log^2n \rfloor$. For $u \in \Z^2$, define the square $S_u :=  (2r)u + [-r,r)^2$, and use the density result (Proposition \ref{prop:4_gandolfi})  of Durrett and Schonmann with a union bound to obtain positive constants $c_1(p,\e)$ and $c_2(p,\e)$ so that the event 
\begin{align}
\cal{A}_n := \left\{ u \in \Z^2,\, S_u \cap [-n,n]^2 \neq \emptyset \implies \left| \frac{ | \cluster \cap S_u| }{ \Leb(S_u) } - \theta_p \right| < \e \right\} 
\end{align}
occurs with probability at least $\prob_p(\cal{A}_n) \geq 1-c_1 \exp(-c_2 \log^2n)$. Given $U \in \cal{U}_n$, let $(\gamma, \pa)$ and $\{(\gamma_j, \pa_j)\}_{j=1}^m$ be pairs of corresponding right-most and interface circuits for $U$, as in Lemma \ref{lem:5_circuit_decomp}. Together, these circuits allow us to form $\vol(U)$, defined in \eqref{eq:5_vol_U_def}. Define two collections of squares:
\begin{align}
\sfS_1 &:= \Big\{ S_u \:: u \in \Z^2,\, S_u \cap \pa \vol(U) \neq \emptyset \Big\} \,,\\
\sfS_2 &:= \Big\{ S_u \:: u \in \Z^2,\, S_u \subset \big( \vol(U) \setminus \pa \vol(U)\big) \Big\} \,,
\end{align}
and let $y \in \vol(U)$. As the $S_u$ form a partition of $\R^2$, it follows that $y$ lives in exactly one $S_u$, which is then either in $\sfS_1$ or $\sfS_2$. If $S_u \in \sfS_1$, there is $u' \in \Z^2$ with $| u - u'|_\infty \leq 1$ so that $S_u$ contains a vertex in $\gamma$ or some $\gamma_j$. In this case, $\text{dist}_\infty(y, U) \leq 4\log^2n$. On the other hand, if $B_u \in \sfS_2$, working within the event $\cal{A}_n$, we find $S_u \cap \cluster \subset U$ is non-empty and hence that $\text{dist}_\infty(y,U) \leq 4\log^2n$. As $U \subset \vol(U)$, it follows from the above observations that $\Haus(U, \vol(U)) \leq \log^4n$, for $n$ sufficiently~large. 

We turn to the density of $U$ within $\vol(U)$, and here we follow the proof of Lemma 5.3 in \cite{BLPR}. Let $\vol(U)_r$ be the union of all squares in $\sfS_1$ and let $\vol(U)^r$ be the union of all squares in $\sfS_1 \cup \sfS_2$, so that $\vol(U)_r \subset \vol(U) \subset \vol(U)^r$. We continue to work within $\cal{A}_n$, and we now assume $U$ is well-proportioned and satisfies $\Leb(\vol(U)) \geq \log^{20}n$. We have
\begin{align}
|U| &\leq | \vol(U)^r \cap \B{C}_\infty | \leq (\theta_p + \e) \Leb( \vol(U)^r) \\
&\leq (\theta_p + \e) \big( \Leb(\vol(U)) + C' \Leb(B_u) \dper(U) \big) \\
&\leq \theta_p \Leb( \vol(U)) (1 + C\e)
\end{align}
for some $C',C >0$ and where $n$ is taken sufficiently large to obtain the last line. The lower bound $|U| \geq \theta_p \Leb( \vol(U)) (1 - C\e)$ follows similarly, finishing the proof. \end{proof}

Given $U \in \cal{U}_n$, we we will build a polygonal approximate from a collection of simple polygonal circuits. It is convenient to introduce the following construction, used in Lemma \ref{lem:5_poly_construct} which is in turn used in the proof of Proposition \ref{prop:5_general_shape_poly} below.

\begin{defn}
Given polygonal curves $\rho, \rho_1, \dots, \rho_m \subset \R^2$, we define the set $\hull(\rho,\rho_1, \dots, \rho_m)$ to be the union of $\rho \cup \rho_1\cup \dots \cup \rho_m$ with
\begin{align}
\left\{ x \in \R^2 \setminus \left(\rho \cup \bigcup_{j=1}^m \rho_j \right) \:: w_\rho(x) - \left( \sum_{j=1}^m w_{\rho_j}(x) \right) \text{ is odd} \right\}\,,
\end{align}
where we recall $w_\rho(x), w_{\rho_j}(x)$ are the winding numbers of these curves about $x$. 
\end{defn}

Note that, in general, $\hull(\rho, \rho_1, \dots, \rho_m)$ is not a polygon, though it is when the curves $\rho,\rho_j$ are in general position. 

\begin{figure}[h]
\centering
\includegraphics[scale=.75]{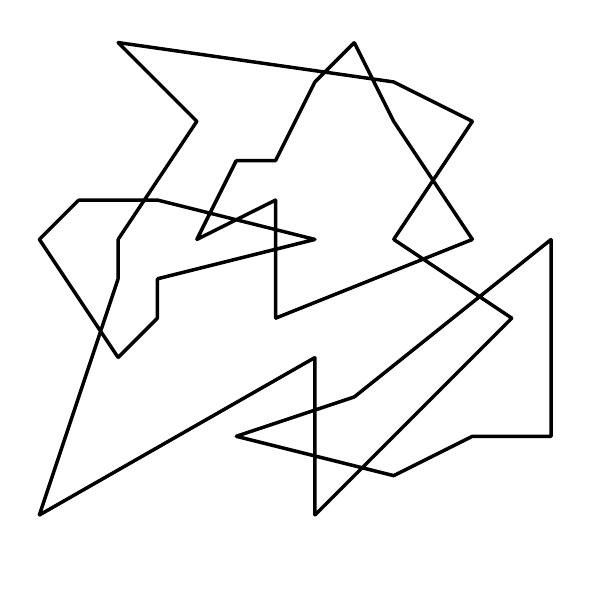}\hspace{5mm}
\includegraphics[scale=.75]{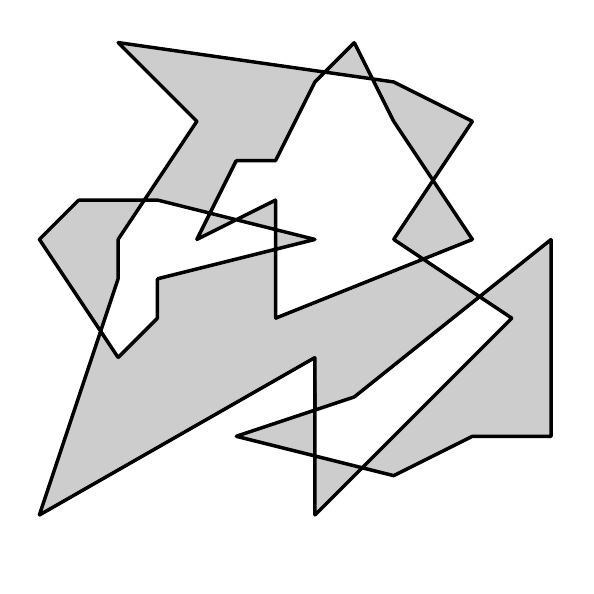}
\caption{On the left, the curves $\rho, \rho_1, \rho_2, \rho_3$. On the right, $\hull(\rho, \rho_1, \dots, \rho_3)$. As these curves are in general position, $\hull(\rho, \rho_1, \dots, \rho_3)$ is a polygon.} 
\label{fig:multi_hull}
\end{figure}

\begin{lem} Let $R \in \cal{R}$ be connected, with $\pa R$ consisting of the Jordan curves $\lambda, \lambda_1, \dots, \lambda_m$. Let $\delta >0$ and let $\rho, \rho_1, \dots, \rho_m \subset [-1,1]^2$ be simple polygonal circuits so that $\Haus(\lambda, \rho) \leq \delta$ and so that $\Haus(\lambda_j, \rho_j) \leq \delta$ for each $j$. We suppose that $\delta$ is small enough so that $\hull(\rho_j)^\circ \cap \hull(\rho)^\circ$ is non-empty for each $j$. There are simple polygonal circuits $\rho', \rho_1', \dots, \rho_m' \subset [-1,1]^2$ so that 
\begin{enumerate}
\item $\Haus(\rho, \rho') \leq \delta$ and $\Haus(\rho_j, \rho_j') \leq \delta$ for each $j$, \label{eq:poly_construct_item_one}
\item $P := \hull(\rho',\rho_1', \dots, \rho_m')$ is a connected polygon,\label{eq:poly_construct_item_two}
\item $\Haus(R, P) \leq 2\delta$ \label{eq:poly_construct_item_three}
\item $\tension(\rho) + \tension(\rho_1) + \dots + \tension(\rho_m) + \delta \geq \tension(\pa P)$. \label{eq:poly_construct_item_four}
\end{enumerate}
\label{lem:5_poly_construct}
\end{lem}

\begin{proof} Using the continuity of the norm $\beta_p$, we may perturb each $\rho, \rho_1, \dots, \rho_m$ to a collection $\rho', \rho_1', \dots, \rho_m'$ of simple polygonal curves in general position with respect to each other satisfying \eqref{eq:poly_construct_item_one} and \eqref{eq:poly_construct_item_four}. Taking $\delta$ smaller if necessary, and using the hypotheses of the lemma, we may execute this perturbation in such a way that $\hull(\rho_j')^\circ \cap \hull(\rho')^\circ$ is non-empty for each $j$. Using this and the transversality of the $\rho', \rho_j'$, it follows that $ \hull(\rho',\rho_1', \dots, \rho_m')$ is a connected polygon, which settles \eqref{eq:poly_construct_item_two} (connectedness can be established by inducting on the number $m$ of polygonal curves $\rho_1', \dots, \rho_m'$). 

We turn our attention to the Hausdorff distance between $R$ and $P := \hull(\rho',\rho_1', \dots, \rho_m')$. Let $x \in R$. If $x \in \pa R$, there is $y \in \pa P$ a distance of at most $2\delta$ from $x$. If $x \in R$ and $x \notin P$, we appeal to the definition of $\hull$ (using winding number) to deduce that $x$ is at most $2\delta$ from $\pa P$. A symmetric argument starting with $x \in P$ settles \eqref{eq:poly_construct_item_three}. \end{proof}

Proposition \ref{prop:5_general_shape_poly} below is our first tool for passing from elements of $\cal{U}_n$ to connected polygons.

\begin{prop} Let $p > p_c(2)$ and let $\e >0$. There are positive constants $c_1(p,\e)$ and $c_2(p,\e)$ so that with probability at least $1 - c_1 \exp(-c_2 \log^2n)$, whenever $U \in \cal{U}_n$ satisfies 
\begin{enumerate}
\item $U$ is well-proportioned,
\item $\Leb(\vol(U)) \geq n^{7/4}$,
\item $|\pa^\infty U| \leq Cn$.
\end{enumerate}
there is a connected polygon $P = P(U) \in \cal{R}$ so that 
\begin{enumerate}
\item $\Haus(U, nP) \leq n^{1/2}$,
\item $\big| |U| - \theta_p\Leb(nP) \big| \leq \e |U|$,
\item $| \pa^n U| \geq (1-\e) \tension(n \pa P)$.
\end{enumerate}
\label{prop:5_general_shape_poly}
\end{prop}

\begin{proof} Let $U \in \cal{U}_n$. Using Lemma \ref{lem:5_circuit_decomp}, form the pairs of right-most and interface circuits $(\gamma, \pa)$ and $\{ (\gamma_j , \pa_j) \}_{j=1}^m$ associated to $U$. We view the interfaces $\pa,\pa_j$ as Jordan curves (via ``corner-rounding," see Remark \ref{rmk:2_corner_round}). Recall that we denoted $\hull(\pa_j) \cap \cluster$ as $\Lambda_j$, and that the $\Lambda_j$ are the finite connnected components of $\cluster \setminus U$. We say $\Lambda_j$ is \emph{large} if $| \Lambda_j | \geq n^{1/2}$ and that it is \emph{small} otherwise. \\

\emph{\B{Step I: (Filling of small components)}}  Let $(\wt{\gamma}_1,\wt{\pa}_1) \dots, (\wt{\gamma}_\ell, \wt{\pa}_\ell)$ enumerate the pairs of right-most circuits and corresponding interfaces associated to the large components $\Lambda_j$. Define 
\begin{align}
R :=  \hull(\pa) \setminus \left( \bigsqcup_{i=1}^\ell \hull(\wt{\pa}_i)^\circ \right) \,,
\end{align}
and let $\wt{U} := R \cap \cluster$ (hence, $R = \vol(\wt{U})$). Observe that $\wt{U}$ is well-proportioned because $U$ is. By construction, $\wt{U}$ is close to $U$ both in $\Haus$-sense and in volume. To see this, observe that the open edge boundaries of each $\Lambda_j$ are disjoint and are each subsets of $\pa^\infty U$. The hypothesis $|\pa^\infty U |\leq Cn$ implies 
\begin{align}
| \wt{U} \setminus U | \leq Cn^{3/2}\,,
\label{eq:5_general_shape_1}
\end{align}
and it is immediate that
\begin{align}
\Haus(U, \wt{U}) \leq n^{1/2} \,.
\label{eq:5_general_shape_2}
\end{align}

\emph{\B{Step II: (Constructing a polygon $P$)}} We use Lemma \ref{lem:5_poly_curve_2} and Lemma \ref{lem:5_poly_construct} to build a suitable polygon from $\wt{U}$. By Corollary \ref{lem:A_new_discrete_iso}, for each large $\wt{\gamma}_i$, we have $|\wt{\gamma}_i| \geq n^{1/8}$ for $n$ sufficiently large, and likewise that $|\gamma | \geq n^{1/8}$. Work within the high probability event from Lemma \ref{lem:5_poly_curve_2} and find simple polygonal circuits $\rho_i \subset [-1,1]^2$ for each large $\wt{\gamma}_i$ so that 
\begin{enumerate}
\item $\Haus(\wt{\pa}_i, n\rho_i) \leq 2n^{1/16}$,
\item $| \frak{b}^n(\wt{\gamma}_i) | \geq (1-\e) \tension( n \rho_i)$, \label{eq:5_general_shape_3}
\end{enumerate}
as well as a polygonal circuit $\rho \subset [-1,1]^2$ corresponding to $\gamma$ with
\begin{enumerate}
\setcounter{enumi}{2}
\item $\Haus(\pa, n\rho) \leq n^{1/16}$,
\item $| \frak{b}^n(\gamma) | \geq (1-\e) \tension( n \rho)$.\label{eq:5_general_shape_4}
\end{enumerate}
In the case that there are no large components, we simply define $P := \hull(\rho)$. In the case that the collection of large components is non-empty, we define $P$ differently below. Using Lemma \ref{lem:5_poly_construct}, find polygonal circuits $\rho', \rho_1', \dots, \rho_\ell' \subset [-1,1]^2$ so that 
\begin{enumerate}
\setcounter{enumi}{4}
\item $\Haus(n\rho, n\rho') \leq n^{1/16}$ and $\Haus(n\rho_i, n\rho_i') \leq n^{1/16}$ for each $i$,
\item $P := \hull(\rho',\rho_1', \dots, \rho_\ell')$ is a connected polygon,
\item $\Haus(R, P) \leq 2 n^{1/16}$\label{eq:5_general_shape_item_7}
\item $\tension(\rho) + \tension(\rho_1) + \dots + \tension(\rho_\ell) + n^{-15/16} \geq \tension(\pa P)$.\label{eq:5_general_shape_item_8}
\end{enumerate}
In either case, we will show the polygon $P \subset [-1,1]^2$ has the desired properties. \\

\emph{\B{Step III: (Controlling $\tension(\pa P)$)}} Within the first case that $P = \hull(\rho)$, we find
\begin{align}
| \pa^n U| \geq |\pa^n \wt{U}| = | \frak{b}^n(\gamma) | \geq (1-\e) \tension(n\pa P) \,,
\label{eq:5_tension_bound}
\end{align}
which is satisfactory. Thus we may suppose the set of large components is non-empty. Let $\al >0$ be as in the statement of Proposition \ref{prop:2_length_comparison} and let 
\begin{align}
\cal{E}_n := \left\{ \exists \gamma \in \bigcup_{\substack{x_0 \in [-n,n]^2\, \cap\, \Z^2 \\ x \in \Z^2}} \cal{R}(x_0,x) \:: n^{1/8} \leq |\gamma| \leq 100n^2\,,\, | \frak{b}(\gamma)| \leq \al |\gamma| \right\}\,,
\label{eq:5_low_prob_event_paths}
\end{align}
so that Proposition \ref{prop:2_length_comparison} with a union bound gives positive constants $c_1(p)$ and $c_2(p)$ so that $\prob_p(\cal{E}_n) \leq c_1 \exp(-c_2 n)$. Work within $\cal{E}_n^c$ for the remainder of the proof, and use the fact that $\frak{b}(\wt{\gamma}_i) = \frak{b}^n(\wt{\gamma}_i)$, along with the bound $|\wt{\gamma}_i | \geq n^{1/8}$:
\begin{align}
(1+\e) |\pa^n \wt{U}| &= (1+\e) \left( | \frak{b}^n(\gamma) | + \sum_{i=1}^\ell | \frak{b}^n(\wt{\gamma}_i) | \right) \,, \\
&\geq | \frak{b}^n(\gamma) | + \sum_{i=1}^\ell | \frak{b}^n(\wt{\gamma}_i) | + n^{1/16} \,,
\label{eq:5_general_shape_step_3}
\end{align} 
for $n$ sufficiently large. Continuing from \eqref{eq:5_general_shape_step_3}, let us use \eqref{eq:5_general_shape_3}, \eqref{eq:5_general_shape_4} and \eqref{eq:5_general_shape_item_8}:
\begin{align}
|\pa^n U| &\geq |\pa^n \wt{U}| \geq \frac{1}{1+\e} \left( |\frak{b}^n(\gamma)| + \sum_{i=1}^\ell |\frak{b}^n(\wt{\gamma}_i)| + n^{1/16} \right) \,,\\
&\geq \frac{1-\e}{1+\e} \left( \tension(n\rho) + \sum_{i=1}^\ell \tension(n\rho_i) + n^{1/16} \right)\,,\\
&\geq \frac{1-\e}{1+\e} \tension(n \pa P) \,,
\label{eq:5_general_shape_5}
\end{align}
so $P$ has the desired properties as far as the surface tension in this case as well. \\

\emph{\B{Step IV: ($\Haus$-closeness of $nP$ and $\wt{U}$)}} Let $\cal{A}_n$ be the high probability event from Lemma \ref{lem:5_volume_for_good}, and work within this event for the remainder of the proof. In the case that the collection of large components is empty, $P = \hull(\rho)$ implies $\Haus(R, nP) \leq n^{1/16}$. As $R = \vol(\wt{U})$, it follows from working within $\cal{A}_n$ that
\begin{align}
\Haus(\wt{U},nP) \leq n^{1/16} + \log^4n\,.
\label{eq:5_dH_closeness_1}
\end{align}
On the other hand, if the collection of large components is non-empty, \eqref{eq:5_general_shape_item_7} implies
\begin{align}
\Haus(\wt{U},nP) \leq 2n^{1/16} + \log^4n\,,
\label{eq:5_dH_closeness_2}
\end{align}
as desired.\\

\emph{\B{Step V: (Controlling the volume of $P$)}} Let $r = \lceil n^{1/16} \rceil$, and for $x \in \Z^d$ let $B_x = x + [-2r,2r]^2$. Let $V(\wt{U})$ denote the vertices of $\Z^2$ contained in the union of paths $\gamma \cup \bigcup_{i=1}^\ell \wt{\gamma}_i$. Observe that, in either construction of $P$, we have
\begin{align}
nP \, \Delta \, R \subset \bigcup_{x \in V(\wt{U}) } B_x \,,
\end{align}
so that 
\begin{align}
\Leb( nP\, \Delta\, R) &\leq 100n^{1/16} \left[\dper(\wt{U})\right] \,, \\
&\leq 100n^{1/16} \left[ \Leb( \vol(\wt{U})) \right]^{2/3} \,, 
\end{align}
as $\wt{U}$ is well-proportioned. As $\wt{U}$ is also sufficiently large and we are working within the event $\cal{A}_n$, we also have $\big| |\wt{U}| - \theta_p \Leb(R)\big| \leq \e \Leb(R)$, thus
\begin{align}
\Leb( nP\, \Delta\, R) &\leq 100n^{1/16} \left[ \frac{ |\wt{U}|}{ \theta_p-\e} \right]^{2/3} \leq \e |\wt{U}| \,, 
\end{align}
for $n$ sufficinently large. It follows that
\begin{align}
\big| |\wt{U}| - \theta_p \Leb(nP)\big| &\leq \big| |\wt{U}| - \theta_p\Leb(R) \big| + \e |\wt{U}| \,,\\
&\leq  \left(\frac{\e}{\theta_p -\e} +\e \right) |\wt{U}| \,.
\end{align}

\emph{\B{Step VI: (Wrapping up)}} Using \eqref{eq:5_general_shape_1}, we have
\begin{align}
\big| |U| - \theta_p \Leb(nP) \big| &\leq \left( \frac{\e}{\theta_p - \e} + \e \right)( |U| + Cn^{3/2} ) + Cn^{3/2}\,,\\
&\leq C' \e |U| \,,
\end{align}
for some $C' >0$ and when $n$ is taken sufficiently large. By \eqref{eq:5_general_shape_2} and either \eqref{eq:5_dH_closeness_1} or \eqref{eq:5_dH_closeness_2}, we also have $\Haus( U,nP) \leq n^{1/2}$ for $n$ sufficiently large. Finally, recall that from either \eqref{eq:5_tension_bound} or \eqref{eq:5_general_shape_5} we have $| \pa^n U| \geq \tfrac{1-\e}{1+\e} \tension( \pa nP)$. The proof is complete upon adjusting $\e$. \end{proof}

We now apply Proposition \ref{prop:5_general_shape_poly} to connected Cheeger optimizers. Let us define
\begin{align}
\cal{G}_n^* := \big\{ G_n \in \cal{G}_n \:: G_n \text{ is connected} \big\}\,.
\end{align}

\begin{prop} Let $p > p_c(2)$. There are positive constants $c_1(p,\e), c_2(p,\e)$ so that for all $n \geq 1$, with probability at least $1- c_1 \exp(-c_2 \log^2n)$, for each $G_n \in \cal{G}_n^*$, there is a connected polygon $P_n \equiv P(G_n,\e) \subset [-1,1]^2$ satisfying
\begin{enumerate}
\item $\Haus(G_n, nP_n) \leq  2n^{1/2} $,
\item $\big| |G_n| - \theta_p\Leb(nP_n) \big| \leq \e |G_n|$,
\item $| \pa^n G_n| \geq (1-\e) \tension(n \pa P_n)$.
\end{enumerate}
\label{prop:5_optimizer_qualities}
\end{prop}

\begin{proof} Let us work within the high probability event from Proposition \ref{prop:5_vol_order_opt} that for some $\al_1 >0$, we have $\min_{G_n \in \cal{G}_n} | G_n| \geq \al_1 n^2$. In conjunction with Proposition \ref{prop:A_surface_order}, we find that $\max_{G_n \in \cal{G}_n} |\pa^n G_n| \leq \al' n$ for some $\al' >0$. As $| \pa^\infty G_n \setminus \pa^n G_n | \leq 8n$ for all $G_n \in \cal{G}_n$, it follows that $\max_{G_n \in \cal{G}_n} |\pa^\infty G_n| \leq \al_2 n$ for some $\al_2 >0$. Fix $G \equiv G_n \in \cal{G}_n^*$, and observe that $G \in \cal{U}_n$. 

We begin by following the proof of Propostion \ref{prop:5_general_shape_poly}: consider the pairs of right-most and interface circuits $(\gamma, \pa)$ and $\{ ( \gamma_j, \pa_j) \}_{j=1}^m$ which give rise to $\vol(G)$ and let $\Lambda_j$ denote $\hull(\pa_j) \cap \cluster$. Say that $\Lambda_j$ is \emph{large} if $|\Lambda_j | \geq n^{1/2}$ and that $\Lambda_j$ is \emph{small} otherwise. Define
\begin{align}
\wt{G} :=\left[ \hull(\pa) \setminus \left( \bigsqcup_{j \::\: \Lambda_j \text{ large}} \hull(\pa_j) \right) \right] \cap \cluster \,,
\end{align}
As in the proof of Propostion \ref{prop:5_general_shape_poly}, we observe $\wt{G}$ is close to $G$ both in $\Haus$-sense and in volume; as $|\pa^\infty G_n | \leq \al_2 n$, we find 
\begin{align}
| \wt{G} \setminus G | \leq \al_2 n^{3/2} \hspace{5mm} \text{and} \hspace{5mm} \Haus(\wt{G},G) \leq n^{1/2}.
\label{eq:5_optimizer_qualities_1}
\end{align}

\emph{\B{Step I: (Controlling $\dper(\wt{G})$)}} The isoperimetric inequality (Corollary \ref{lem:A_new_discrete_iso}) implies $|\gamma_j | \geq n^{1/8}$ for any $\Lambda_j$ which is large. Likewise, because $|G| \geq \al_1 n^2$, we also have $|\gamma | \geq n^{1/8}$. Let $\al >0$ be as in the statement of Proposition \ref{prop:2_length_comparison} and let $\cal{E}_n$ be the event defined in \eqref{eq:5_low_prob_event_paths}. Work within the high probability event $\cal{E}_n^c$ for the remainder of the proof, so that $|\frak{b}(\gamma)| \geq \al |\gamma|$ and for each large $|\Lambda_j|$ we find $|\frak{b}(\gamma_j) | \geq \al |\gamma_j|$. It follows that 
\begin{align}
\dper(\wt{G}) \leq \frac{\al_2}{\al}n\,.
\end{align}

\emph{\B{Step II: (Showing $\Leb(\vol(\wt{G}))$ is on the order of $n^2$)}} By construction, for some $C >0$,
\begin{align}
\Leb(\vol(\wt{G})) &\geq | \vol (\wt{G}) \cap \Z^2 | - C \dper(\wt{G}) \,,\\
&\geq |G| - C\dper(\wt{G}) \geq \frac{\al_1}{2} n^2\,,
\end{align}
for $n$ sufficiently large. We thus conclude that $\wt{G}$ is well-proportioned and satisfies $\Leb(\vol(\wt{G})) \geq n^{7/4}$ when $n$ is large enough. Moreover, $\pa^\infty \wt{G} \subset \pa^\infty G$, so that $| \pa^\infty \wt{G} | \leq \al_2 n$, and $\wt{G}$ satisfies all necessary prerequisites of Proposition \ref{prop:5_general_shape_poly}. \\

\emph{\B{Step III: (Building a polygon)}} Work within the high probability event from Propostion \ref{prop:5_general_shape_poly}, use \eqref{eq:5_optimizer_qualities_1} and the fact that $\pa^n \wt{G} \subset \pa^n G$ to obtain a polygon $P \equiv P(G,\e) \subset [-1,1]^2$ with
\begin{enumerate}
\item $\Haus(G, nP) \leq 2n^{1/2}$,
\item $\big| |G| - \theta_p\Leb(nP) \big| \leq 2\e |G|$,
\item $| \pa^n G| \geq (1-\e) \tension(n \pa P)$,
\end{enumerate}
where we have taken $n$ sufficiently large to obtain the second item directly above. The proof is complete. \end{proof}

\subsection{\small\textsf{Proofs of main theorems}} We begin by proving a precursor to Theorem \ref{thm:main_shape} for connected Cheeger optimizers.

\begin{prop} Let $p > p_c(2)$ and let $\e >0$. There are positive constants $c_1(p,\e)$ and $c_2(p,\e)$ so that for all $n\geq 1$, with probability at least $1 -c_1 \exp(-c_2 \log^2n)$, we have
\begin{align}
\max_{G_n \in \cal{G}_n^*} \Haus(n^{-1} G_n, \cal{R}_p) \leq \e \,.
\end{align}
We emphasize that the maximum directly above runs over $\cal{G}_n^*$. 
\label{prop:5_limit_shape_precursor}
\end{prop}

\begin{proof} Let $\e >0$, and define the event
\begin{align}
\cal{E}^{(n)}(\e) := \Big\{ \exists G_n \in \cal{G}_n^* \:: \Haus( n^{-1}G_n, \cal{R}_p ) > \e \Big\} 
\end{align}
Let $\e' >0$ to be determined later, and let $\cal{A}_1^{(n)}(\e')$ be the event from Proposition \ref{prop:5_optimizer_qualities} that for each $G_n \in \cal{G}_n^*$, there is a connected polygon $P_n \subset [-1,1]^2$ so that 
\begin{enumerate}
\item $\Haus(G_n, nP_n) \leq 2n^{1/2} $, \label{eq:5_limit_shape_precursor_item_1}
\item $\big| |G_n| - \theta_p\Leb(nP_n) \big| \leq \e' |G_n|$, \label{eq:5_limit_shape_precursor_item_2}
\item $| \pa^n G_n| \geq (1-\e') \tension(n \pa P_n)$,\label{eq:5_limit_shape_precursor_item_3}
\end{enumerate}
Let us first give an upper bound on $\Leb(P_n)$ within the event $\cal{A}_1^{(n)}(\e')$ and another high probability event. Let
\begin{align}
\cal{A}_2^{(n)}(\e') := \left\{ \frac{ |\giant|}{(2n)^2} \in \Big( (1-\e')\theta_p, (1+\e')\theta_p \Big) \right\} \,,
\end{align}
so that by Proposition \ref{prop:4_gandolfi} and Proposition \ref{prop:4_benj_mo}, there are positive constants $c_1(p,\e'), c_2(p,\e')$ with $\prob(\cal{A}_2^{(n)}(\e')^c) \leq c_1 \exp(-c_2 \log^2n)$. Within the intersection $\cal{A}_1^{(n)}(\e') \cap \cal{A}_2^{(n)}(\e')$ and using \eqref{eq:5_limit_shape_precursor_item_2}, we have
\begin{align}
\max_{G_n \in \cal{G}_n^*} \Leb(P_n) \leq 2 ( 1 +\e')^2 \,,
\end{align}
and let us choose $\al = \al(\e') >0$ so that $2+ \al = 2(1+\e')^2$. Recall that Corollary \ref{coro:3_stability_final} tells us there is $\delta = \delta(\e) >0$ so that when $R \in \cal{R}$ is connected with $\Leb(R) \leq 2 +\al$ and $\Haus(R, \cal{R}_p^{(2+\al)} ) > \e /100$, we have
\begin{align}
\frac{\tension(\pa R)}{\Leb(R)} > \vp_p^{(2+\al)} + \delta \,.
\label{eq:5_final_tension_bound}
\end{align}
We now take $\e'$ small enough so that (using Lemma \ref{lem:3_opt_close}), we have $\Haus( \cal{R}_p^{(2+\al)}, \cal{R}_p ) \leq \e/4$. Thus, for this $\e'$, within $\cal{E}_n(\e) \cap \cal{A}_1^{(n)}(\e') \cap \cal{A}_2^{(n)}(\e')$ and for $n$ sufficiently large (using \eqref{eq:5_limit_shape_precursor_item_1}), the following event occurs
\begin{align}
\left\{ \Haus(P_n, \cal{R}_p^{(2+\al)} ) > \e/4 \right\} \,,
\end{align}
so that by \eqref{eq:5_final_tension_bound}, \eqref{eq:5_limit_shape_precursor_item_2} and \eqref{eq:5_limit_shape_precursor_item_3}, we have
\begin{align}
n\Chee &\geq (1-\e')^2 \theta_p^{-1} \frac{\tension(\pa P_n)}{\Leb(P_n) } \,, \\
&\geq (1-\e')^2 \theta_p^{-1} \Big[  \vp_p^{(2+\al)} + \delta \Big] \,,
\end{align}
within $\cal{E}_n(\e) \cap \cal{A}_1^{(n)}(\e') \cap \cal{A}_2^{(n)}(\e')$. Working within this intersection, we use Corollary \ref{coro:3_duality} to deduce 
\begin{align}
n\Chee &\geq (1-\e')^2 \theta_p^{-1} \left( \frac{2-\al}{2+\al}\vp_p^{(2-\al)} + \delta \right) \,, \\
&\geq (1-\e')^2 \theta_p^{-1} \left( \frac{2-\al}{2+\al} \vp_p + \delta \right) \,, \\
&\geq \theta_p^{-1} \left( \vp_p + \delta /2 \right) \,,
\end{align}
where we have taken $\e'$ sufficiently small (depending on $\delta$ and hence $\e$) to obtain the last line, and where we emphasize the cruciality that $\delta$ does not depend on $\e'$. Thus,
\begin{align}
\prob_p( \cal{E}_n(\e) ) \leq \prob_p( \cal{A}_1^{(n)}(\e')^c) + \prob_p( \cal{A}_2^{(n)}(\e')^c) + \prob_p \left( n \Chee \geq  \theta_p^{-1} \left( \vp_p + \delta /2 \right) \right)
\label{eq:5_limit_shape_precursor_final}
\end{align}
We have established that $\cal{A}_1^{(n)}(\e')^c$ and $\cal{A}_2^{(n)}(\e')^c$ are low-probability events; we bound the last term in \eqref{eq:5_limit_shape_precursor_final} using Theorem \ref{thm:4_chee_upper_final} to complete the proof.
\end{proof}

\emph{\B{Proof of Theorem \ref{thm:main_asym}.}} Let $\delta >0$, and let $\cal{A}_1^{(n)}(\delta)$ and $\cal{A}_2^{(n)}(\delta)$ be the high-probability events from the proof of Proposition \ref{prop:5_limit_shape_precursor} for the parameter $\delta$ in place of $\e'$. Within the intersection $\cal{A}_1^{(n)}(\delta) \cap \cal{A}_2^{(n)}(\delta)$, we have for each $G_n \in \cal{G}_n^*$ a connected polygon $P_n \subset [-1,1]^2$ satisfying
\begin{enumerate}
\item $\Leb(P_n) \leq 2(1+\delta)^2$\,,
\item $\big| |G_n| - \theta_p\Leb(nP_n) \big| \leq \delta |G_n|$\,,
\item $| \pa^n G_n| \geq (1-\delta) \tension(n \pa P_n)$\,,
\end{enumerate}
and as before we define $\al = \al(\delta) >0$ so that $2(1+\delta)^2 = 2+\al$. Thus, within $\cal{A}_1^{(n)}(\delta) \cap \cal{A}_2^{(n)}(\delta)$ we have
\begin{align}
n\Chee &\geq (1-\delta)^2 \frac{ \tension(\pa P_n) }{ \theta_p \Leb(P_n) }\,, \\
&\geq (1-\delta)^2 \frac{\vp_p^{(2 +\al)} }{\theta_p}\,,\\
&\geq \frac{(1-\delta)^2(2-\al)}{2+\al} \frac{\vp_p}{\theta_p} \,,
\end{align}
where we have used Corollary \ref{coro:3_duality} and the fact that $\vp_p^{(2-\al)} \geq \vp_p$ to obtain the last line. Thus, for $\e >0$, we may take $\delta$ and hence $\al$ sufficiently small so that within $\cal{A}_1^{(n)}(\delta) \cap \cal{A}_2^{(n)}(\delta)$, we have $n \Chee \geq (1-\e) ( \vp_p / \theta_p)$. Using Theorem \ref{thm:4_chee_upper_final}, we then conclude that for all $n \geq 1$, there are positive constants $c_1(p,\e)$ and $c_2(p,\e)$ so that with probability at least $1- c_1 \exp(-c_2 \log^2n)$, we have
\begin{align}
(1+\e)\vp_p \geq n\Chee \geq (1-\e) \vp_p.
\end{align}
We apply Borel-Cantelli to complete the proof.\hfill\qed

\vspace{4mm}

\emph{\B{Proof of Theorem \ref{thm:main_shape}.}} Our strategy is to show that each $G_n \in \cal{G}_n^*$ is large. By Lemma \ref{lem:3_strict_mon}, we have $\vp_p^{(7/4)} > \vp_p$. Let $\e >0$ be small enough so that $\vp^{(7/4)} > (1+\e) \vp_p$, and choose $\delta$ depending on this $\e$ so that 
\begin{align}
(1-\delta)^2 \vp_p^{(7/4)} \geq (1+ \e) \vp_p\,.
\label{eq:5_final_delta_ep}
\end{align}
For this $\delta$, work within the intersection $\cal{A}_1^{(n)}(\delta) \cap \cal{A}_2^{(n)}(\delta)$, the events introduced in the proof of Proposition \ref{prop:5_limit_shape_precursor}, so that for each $G_n \in \cal{G}_n^*$, there is a connected polygon $P_n \subset [-1,1]^2$ with
\begin{enumerate}
\item $\Haus(G_n, nP_n) \leq 2n^{1/2} $, \label{eq:5_final_proof_item_1}
\item $\big| |G_n| - \theta_p\Leb(nP_n) \big| \leq \delta |G_n|$, \label{eq:5_final_proof_item_2}
\item $| \pa^n G_n| \geq (1-\delta) \tension(n \pa P_n)$,\label{eq:5_final_proof_item_3}
\end{enumerate}
Thus by \eqref{eq:5_final_proof_item_2}, \eqref{eq:5_final_proof_item_3} and \eqref{eq:5_final_delta_ep}
\begin{align}
\cal{A}_1^{(n)}(\delta) \cap \cal{A}_2^{(n)}(\delta) \cap \left\{ \exists G_n \in \cal{G}_n^* \:: \Leb(P_n) \leq 7/4 \right\} &\subset \left\{ n \Chee \geq (1-\delta)^2 \vp_p^{(7/4)} \right\} \,, \\
&\subset \left\{ n\Chee \geq (1+\e) \vp_p \right\}\,.
\label{eq:5_final_event}
\end{align}
Let us write $\cal{F}_n(\e)$ for the complement of the event in \eqref{eq:5_final_event}. Theorem \ref{thm:4_chee_upper_final} tells us $\cal{F}_n(\e)$ occurs with high probability, so that on the intersection $\cal{A}_1^{(n)}(\delta) \cap \cal{A}_2^{(n)}(\delta) \cap \cal{F}_n(\e)$, we have
\begin{align}
\min_{G_n \in \cal{G}_n^*} \Leb(P_n) > 7/4 \,,
\end{align}
and hence by \eqref{eq:5_final_proof_item_2},
\begin{align}
\min_{G_n \in \cal{G}_n^*} |G_n| \geq \frac{1}{1+\delta} \theta_p \left(\frac{7}{4}\right)n^2 \,.
\label{eq:5_final_lowerbound}
\end{align}
As we are working within $\cal{A}_2^{(n)}(\delta)$, we also have $|\giant| \leq 4n^2 \theta_p(1+\delta)$, so that from \eqref{eq:5_final_lowerbound} and by taking $\delta$ smaller if necessary, we find 
\begin{align}
\min_{G_n \in \cal{G}_n^*} |G_n| \geq \left(\frac{5}{16}\right) |\giant| \,.
\label{eq:5_final_lowerbound_2}
\end{align}
The inequality $\tfrac{a+b}{c+d} \geq \min\left( \tfrac{a}{c}, \tfrac{b}{d}\right)$ tells us that each $G_n \in \cal{G}_n$ is a disjoint union of elements of $\cal{G}_n^*$. The constraint $|G_n| \leq |\giant|/2$ and \eqref{eq:5_final_lowerbound_2} tell us that 
\begin{align}
\cal{A}_1^{(n)}(\delta) \cap \cal{A}_2^{(n)}(\delta) \cap \cal{F}_n(\e) \subset \Big\{ \cal{G}_n^* \equiv \cal{G}_n \Big\} \,.
\end{align}
Thus, on the intersection of $\cal{A}_1^{(n)}(\delta) \cap \cal{A}_2^{(n)}(\delta) \cap \cal{F}_n(\e)$ and the high-probability event from Proposition \ref{prop:5_limit_shape_precursor}, we find that there are positive constants $c_1(p,\e)$ and $c_2(p,\e)$ so that for each $n \geq 1$, with probability at least $1 -c_1 \exp(-c_2 \log^2n)$, we have
\begin{align}
\max_{G_n \in \cal{G}_n} \Haus(n^{-1} G_n, \cal{R}_p) \leq \e \,,
\end{align}
where we emphasize the above maximum now runs over all of $\cal{G}_n$. The proof is complete upon applying Borel-Cantelli. \hfill\qed

\appendix
{\large\section{\B{Percolation inputs and miscellany}}\label{sec:appendix}}

Recall that $\cal{U}_n$ denotes the connected subgraphs of $\cluster \cap [-1,1]^2$ which are defined by their vertex set. For $U \in \cal{U}_n$, Lemma \ref{lem:5_circuit_decomp} furnishes pairs of right-most circuits and corresponding interfaces $(\gamma, \pa), (\gamma_1, \pa_1), \dots, (\gamma_m, \pa_m)$ which ``carve" $U$ out of $\cluster$. Recall that we used these pairs to define the value $\dper(U)$ in \eqref{eq:5_dper_def} and the set $\vol(U)$ in \eqref{eq:5_vol_U_def}. Recall that we identify the interfaces $\pa, \pa_1, \dots, \pa_m$ with simple closed curves, see Remark \ref{rmk:2_corner_round}.\\

\begin{lem} There is $c >0$ so that for all $n\geq 1$ and for all $U \in \cal{U}_n$, 
\begin{align}
\dper(U) \geq c\Leb( \vol(U))^{1/2}\,.
\end{align}
\label{lem:A_new_discrete_iso}
\end{lem}

\begin{proof} Using the correspondence of Proposition \ref{prop:2_interface_path}, we find constants $c_1, c_2 >0$ so that whenever $\gamma'$ is a right-most circuit with corresponding interface $\pa'$, we have
\begin{align}
c_1 |\gamma' | \leq {\len(\pa')} \leq c_2 | \gamma'| \,,
\label{eq:A_new_discrete_iso_1}
\end{align}
where we view $\pa'$ as a simple circuit in $\R^2$. As the circuits $\pa, \pa_1, \dots, \pa_m$ make up the boundary of the set $\vol(U)$, the standard Euclidean isoperimetric inequality gives $c >0$ so that 
\begin{align}
\len(\pa) + \sum_{i=1}^m \len(\pa_i) \geq c \Leb( \vol(U) )^{1/2} \,.
\label{eq:A_new_discrete_iso_2}
\end{align}
The proof is complete upon combining \eqref{eq:A_new_discrete_iso_1} with \eqref{eq:A_new_discrete_iso_2}. \end{proof}



The next three results are more general percolation inputs. The following result of Durrett and Schonmann (\cite{Durrett_Schonmann} Theorems 2 and 3) allows us to control the density of the infinite cluster within large boxes. 

\begin{prop} Let $p > p_c(2)$, let $\e >0$ and let $r > 0$, and let $B_r \subset \R^2$ be a translate of $[-r,r)^2$. There are positive constants $c_1, c_2$ depending on $p$ and $\e$ so that
\begin{align}
\prob_p \left( \frac{ |\cluster \cap B_r| }{ (2r)^2 } \notin ( \theta_p - \e, \theta_p +\e) \right) \leq c_1 \exp\Big( -c_2 n \Big) \,.
\end{align}
\label{prop:4_gandolfi}
\end{prop}

The next result, due to Benjamini and Mossel, allows us to pass from $\wt{\B{C}}_n = \cluster \cap [-n,n]^2$ to $\giant$ (see Proposition 1.2 of \cite{BenjMo} and Lemma 5.2 of \cite{BLPR}). 

\begin{prop} Let $p > p_c(2)$. There is a positive constant $c(p)$ such that for all $n \geq 1$, with probability at least $1 - \exp (- C \log^2 n) $, and for any $n' \leq n - \log^2 n$, we have
\begin{align}
\cluster \cap [-n',n']^2 = \giant \cap [-n',n']^2\,.
\end{align}

\label{prop:4_benj_mo} 
\end{prop}

Finally we need Proposition A.2 of \cite{BBHK}, which we state in dimension two only.

\begin{prop} Let $p > p_c(2)$. There are positive constants $c_1(p), c_2(p)$ and $\wt{\al}(p)$ so that for all $t >0$,
\begin{align}
\prob_p( \exists \Lambda \subset \cluster, \om\text{-connected}, 0 \in \Lambda, |\Lambda| \geq t^2, |\pa^\infty \Lambda | < \wt{\al} |\Lambda|^{1/2} ) \leq c_1 \exp(-c_2 t) \,.
\end{align}

\label{prop:A_BBHK}
\end{prop}

\bibliographystyle{plain}
\bibliography{gold_sources}
\nocite{*}

\end{document}